\documentclass[twoside,11pt]{amsart}
\usepackage{amsfonts,amssymb,amsmath,amsthm}
\usepackage{hyperref}
\usepackage{tabularx}           
\usepackage{booktabs}           
\usepackage{array}     
\usepackage{xltabular}

\begin{document}
	
	\title[Normalized quadratic extensions of pointed Hopf algebras]
	{Normalized quadratic extensions of pointed Hopf algebras}
	
	\author{Rongchuan Xiong}
	\address{Department of Mathematics, Changzhou University, Changzhou 213164, China}
	\email{rcxiong@foxmail.com}
	
	\maketitle
	
	\newcommand{\K}{\mathbb{K}}
	\newcommand{\Z}{\mathbb{Z}}
	\newcommand{\G}{\mathbf{G}}
	\newcommand{\Pp}{\mathcal{P}}
	\newcommand{\cL}{\mathcal{L}}
	\newcommand{\gr}{\operatorname{gr}}
	\newcommand{\id}{\operatorname{id}}
	\newcommand{\Aut}{\operatorname{Aut}}
	\newcommand{\Hom}{\operatorname{Hom}}
	\newcommand{\End}{\operatorname{End}}
	\newcommand{\HH}{\operatorname{H}}
	\newcommand{\im}{\operatorname{im}}
	\newcommand{\Gau}{\operatorname{Gau}}
	\newcommand{\Char}{\operatorname{char}}
	
	\newtheorem{thm}{Theorem}[section]
	\newtheorem{lem}[thm]{Lemma}
	\newtheorem{pro}[thm]{Proposition}
	\newtheorem{cor}[thm]{Corollary}
	\newtheorem{rmk}[thm]{Remark}
	\newtheorem{defn}[thm]{Definition}
	
	\begin{abstract}
		Let $L$ be a finite-dimensional pointed Hopf algebra over an
		algebraically closed field. We establish an intrinsic characterization
		of index-two extensions: every Hopf algebra $H$ containing $L$ as a
		Hopf subalgebra and satisfying $H_0=L_0$ and
		$\dim H=2\dim L$ is a normalized quadratic extension of $L$.
		For a fixed support $(g,h)$, such extensions are described by a
		coalgebra Hochschild $2$-cocycle together with Ore-type data
		$(\sigma,\delta,u,v)$ satisfying explicit compatibility conditions.
		Their isomorphism classes are parameterized by the orbits of gauge
		transformations and Hopf automorphisms of $L$. As an application, we
		classify non-connected pointed Hopf algebras of dimension $16$ in
		characteristic $2$ with one-dimensional infinitesimal braiding whose
		diagram is not a Nichols algebra.
	\end{abstract}
	
	\section{Introduction}\label{sec:intro}
	
	The classification of finite-dimensional pointed Hopf algebras over
	an algebraically closed field of positive characteristic is a basic
	problem in Hopf algebra theory. For such an algebra $H$, the
	coradical filtration gives rise to an associated graded Hopf algebra
	$\gr H\cong R\#\mathbb K[\G(H)]$, where $R$ is a connected graded braided
	Hopf algebra. The algebra $R$ is called the diagram of $H$, and its
	degree-one part $R(1)$ is the infinitesimal braiding. The subalgebra
	of $R$ generated by $R(1)$ is the Nichols algebra
	$\mathcal B(R(1))$.
	
	The lifting method of Andruskiewitsch and Schneider \cite{AS98b}
	uses these data to classify pointed Hopf algebras. When $R$ equals
	$\mathcal B(R(1))$, the classification of $H$
	reduces to the classification of $R$ and to lifting the relations in
	the bosonization $R\#\K[\G(H)]$. When $R$ strictly contains $\mathcal B(R(1))$, every minimal generating
	set of $H$ over $\K[\G(H)]$ contains elements of coradical filtration
	degree at least two, called \emph{non-primitive generators},
	which are neither group-like nor skew-primitive.
	For such generators one must determine their coproducts and relations. In positive characteristic, such
	generators already occur in low dimensions, and their treatment
	involves additional coproduct and relation data beyond the Nichols
	algebra part.
	
	Ore extensions provide another construction of pointed Hopf
	algebras. Beattie, Dăscălescu and Grünenfelder used Ore
	extensions to construct pointed Hopf algebras \cite{Beattie2000}.
	Panov introduced the notion of Hopf Ore extension, where the
	comultiplication of the adjoined variable is linear in that variable
	with group-like coefficients and no extra cocycle term \cite{Panov2003}.
	You, Wang and Chen generalized this construction by allowing a cocycle term in the
	comultiplication, derived necessary and sufficient conditions for such
	an Ore extension to be a Hopf algebra, and classified generalized
	Hopf--Ore extensions for enveloping algebras of several low-dimensional
	Lie algebras \cite{You2018}. 
	
	For a pointed coalgebra, \v{S}tefan and van Oystaeyen related
	coalgebra Hochschild cohomology with coefficients in one-dimensional
	bicomodules to the coradical filtration \cite{SO}. This yields a
	cohomological framework for studying the coproducts of non-primitive
	generators in a minimal generating set. The perspective is closely
	related to techniques that have appeared in the classification of
	connected Hopf algebras of dimensions $p^2$ \cite{W1} and $p^3$
	\cite{W2,NWW1,NWW2}. Wang, Zhang and Zhuang introduced the terminology
	of primitive cohomology for Hopf algebras and developed this viewpoint
	further in their classification of rank-one pointed Hopf algebras over
	an algebraically closed field of characteristic zero \cite{WZZ}.
	Brown and Zhang used primitive cohomology to compute comultiplication
	corrections in their classification of 2-step iterated Hopf Ore
	extensions \cite{Brown2022}. Related cohomological techniques have also
	appeared in the classification of pointed Hopf algebras of dimensions
	$p^3$ \cite{NW}, $p^2q$ and $pq^2$ \cite{X17,X23a}. Such a viewpoint
	suggests that the coproduct of a higher-degree generator should be
	treated as a cohomology class together with the corresponding
	multiplication data.

	In this paper, we develop a general framework for extensions of finite-dimensional pointed Hopf algebras by one additional generator. We start from a fixed finite-dimensional pointed Hopf
	algebra $L$ and consider Hopf algebras $H$ that contain $L$ as a Hopf
	subalgebra and are generated by $L$ and one additional element $w$,
	which satisfies
	\[
	\Delta(w)=w\otimes g+h\otimes w+\omega,
	\qquad g,h\in \G(L),\quad \omega\in L\otimes L.
	\]
	Coassociativity makes $\omega$ a coalgebra Hochschild $2$-cocycle
	with coefficients in the one-dimensional bicomodule $M_{g,h}$. This viewpoint includes ordinary Hopf Ore extensions as a special case when the coproduct correction is trivial.
	Replacing $w$ by $w+b$, with $b\in L$, changes $\omega$ by a
	coboundary; thus only the cohomology class of $\omega$ in
	$\HH^2_{g,h}(L)$ is an invariant of the coproduct of $w$.
	The multiplication requires additional relations, and these must satisfy
	both associativity and compatibility with the coproduct.
	
	We consider normalized quadratic extensions, that is, extensions for
	which $\{1,w\}$ is a free left $L$-basis and
	\[
	wa=\sigma(a)w+\delta(a)\quad(a\in L),\qquad w^2=u+vw.
	\]
	Here $\sigma$ is a unital algebra endomorphism and $\delta$ is a
	$(\sigma,1)$-derivation. The admissible tuples
	$(g,h,\omega,\sigma,\delta,u,v)$ form the set
	$\mathcal D(L;g,h)$ of Definition~\ref{def:extension-data}.
	Its equations express normalization, associativity, and compatibility
	with the coproduct and counit. Let $\Gamma_{g,h}$ be the group
	generated by changes $w\mapsto\alpha w+b$, with
	$\alpha\in\mathbb K^\times$ and $b\in L^+$, and automorphisms of $L$
	fixing $g,h$, with multiplication as in
	Definition~\ref{def:equivalence}.

	\begin{thm}\label{thm:intro-extensions}
		Let $L$ be a finite-dimensional pointed Hopf algebra over $\K$, and
		fix $g,h\in \G(L)$. Every tuple in $\mathcal D(L;g,h)$ defines a
		normalized quadratic extension with specified support $(g,h)$, of
		dimension $2\dim L$ and with coradical $L_0$. Reconstruction induces
		a bijection
		\[
		\mathcal D(L;g,h)/\Gamma_{g,h}
		\longleftrightarrow
		\left\{
		\begin{array}{c}
			\text{normalized quadratic extensions with specified support }(g,h)\\
			\text{up to isomorphisms as in Definition~\ref{def:normalized}}
		\end{array}
		\right\}.
		\]
	\end{thm}
	Thus the classification problem splits into two parts: the cohomological classification of possible coproduct corrections and the algebraic classification of the compatible multiplication data. By
	Lemma~\ref{lem:generator-rigidity}, each isomorphism of normalized
	extensions with the same specified support sends the chosen generator
	to $\alpha w'+b$, with $\alpha\in\mathbb K^\times$ and $b\in L^+$;
	this is the change of generator appearing in the gauge group. If the
	support is not fixed, group-like multiples must be added as well
	(see Remark~\ref{rmk:unmarked-supports}).

	The following result shows that this construction gives an intrinsic characterization of all index-two extensions with the same coradical.
	\begin{thm}\label{thm:intro-recognition}
		Let $H$ be a finite-dimensional pointed Hopf algebra, and let
		$L\subset H$ be a Hopf subalgebra. Then $H$ is a normalized
		quadratic extension of $L$ if and only if
		\[
		H_0=L_0
		\qquad\text{and}\qquad
		\dim H=2\dim L.
		\]
	\end{thm}
	
	Thus, for a fixed pointed Hopf algebra $L$, the isomorphism classes
	of all finite-dimensional pointed Hopf algebras $H$ with $L\subset H$,
	$H_0=L_0$, and $\dim H=2\dim L$ are parameterized by the orbit spaces
	\[
	\bigsqcup_{g,h\in G(L)}\mathcal D(L;g,h)\Big/\sim,
	\]
	where $\sim$ is the equivalence relation generated by gauge
	transformations, Hopf automorphisms of $L$, and group-like multiples
	of the generator, as in Theorem~\ref{thm:classification-index-two}.
	
	We apply this construction in characteristic $2$ to non-connected
	pointed Hopf algebras of dimension $16$ whose infinitesimal braiding
	is one-dimensional and whose diagram is not a Nichols algebra.
	The case where the diagram is a Nichols algebra was classified in
	\cite{X23}. Write $U_n(\mathbb K)=\{a\in\mathbb K^\times:a^n=1\}$,
	acting on $\K$ by multiplication.
	
	\begin{thm}\label{thm:intro-dimension16}
		Assume $\Char\K=2$. Let $H$ be a non-connected pointed
		Hopf algebra of dimension $16$ with one-dimensional infinitesimal
		braiding and non-Nichols diagram $R$. Then $\dim R\in\{4,8\}$.
		The Hopf isomorphism classes are exactly the entries of
		Tables~\ref{tab:isolated-classes} and \ref{tab:parameter-families},
		with the identifications stated there. More precisely:
		\begin{enumerate}
			\item If $\dim R=4$, there are $15$ individual representatives,
			three one-parameter families with parameter set $\mathbb K$, and one
			two-parameter family with parameter set $\mathbb K^2/\sim$, where
			$\sim$ is generated by
			\[
			(t,u)\longmapsto(t+1,u),\qquad
			(t,u)\longmapsto(t^{-1},u/t^3)\quad(t\ne0).
			\]
			\item If $\dim R=8$, there are nine individual representatives and
			four one-parameter families. Their parameter sets are, respectively,
			$\mathbb K/U_3(\mathbb K)$, $\mathbb K$, $\mathbb K/U_5(\mathbb K)$, and
			$\mathbb K/U_3(\mathbb K)$ for $\mathcal N_1$, $\mathcal N_2$, $\mathcal N_3$, and
			$\mathcal M_3$.
		\end{enumerate}
		There are no identifications between distinct families, between
		individual representatives, or between these two parts of the list.
		Every entry in the tables satisfies the hypotheses above.
	\end{thm}
	
	The full statement is Theorem~\ref{thm:dim16-classification}, with the parameter
	sets written as orbit spaces. The full presentations and supports
	are specified in the two tables. The list contains $24$
	individual representatives and eight parameter families.
	For $\dim R=4$, the Hopf subalgebra
	$L=\langle G(H),x\rangle$ has dimension $8$ and is preserved by every
	Hopf isomorphism; adjoining $y$ gives $H$. For $\dim R=8$, the Hopf
	subalgebra $L=\langle H_2\rangle=H_3$ has dimension $8$ and is
	preserved by every Hopf isomorphism; adjoining $z$ gives $H$. The
	relevant second cohomology groups and their homogeneous degrees
	determine the possible coproduct structures. Comparison of
	coproducts restricts the commutators and squares to explicit lower
	filtration terms. The associativity equations and the Diamond Lemma
	\cite{B} give the remaining restrictions and verify the dimensions.
	Since every Hopf isomorphism preserves $L$, the explicit changes of
	generators determine the Hopf isomorphism classes, including all
	residual parameter identifications.
	
	The paper is organized as follows. Section~\ref{sec:prelim} collects
	basic facts on coradical filtrations, diagrams, infinitesimal
	braidings, and coalgebra Hochschild cohomology.
	Section~\ref{sec:fixed-base} develops the framework of simple
	extensions and proves Theorems~\ref{thm:intro-extensions} and \ref{thm:intro-recognition}.
	Section~\ref{sec:classification} states
	Theorem~\ref{thm:intro-dimension16} and presents the classification
	tables. Its proof is completed in Section~\ref{sec:case-dimR4} for
	the case $\dim R=4$ and in Section~\ref{sec:case-dimR8} for the case
	$\dim R=8$.

	\section{Preliminaries}\label{sec:prelim}
	\subsection*{Notations}
	Throughout, $\K$ is an algebraically closed field. All vector spaces,
	tensor products, and algebras are over $\K$, and all Hopf algebras are
	finite-dimensional unless otherwise stated.
	
	For a Hopf algebra $A$, we write $\Delta_A$, $\epsilon_A$, and $S_A$
	for its coproduct, counit, and antipode; when no confusion arises, we
	omit the subscript. Sweedler notation
	$\Delta(a)=\sum a_{(1)}\otimes a_{(2)}$ will be used.
	
	For a coalgebra $C$, $\G(C)$ denotes the set of group-like elements
	of $C$. For an algebra $A$, we set $[a,b]=ab-ba$; when $A$ is a Hopf
	algebra, we write $A^+=\ker\epsilon_A$. For  $n\in\mathbb N$, $\Z_n$ denotes the cyclic group of
	order $n$ and $U_n(\K)=\{a\in\K^\times:a^n=1\}$.
	
	Other standard symbols: $\Aut$, $\End$, $\Hom$, $\id$, $\im$, and
	$\Char$ denote automorphism, endomorphism, homomorphism, identity map,
	image, and characteristic, respectively. The multiplicative group of
	$\K$ is written $\K^\times$. We refer to \cite{R11} for the general theory of Hopf algebras.
	
	\subsection{Coalgebras and the coradical filtration}
	
	For a coalgebra $C$, its coradical $C_0$ is the sum of its simple
	subcoalgebras. The coradical filtration $\{C_n\}_{n\ge0}$ is defined
	recursively by
	\[
	C_n=\Delta^{-1}(C_0\otimes C+C\otimes C_{n-1}),\quad n\geq 1.
	\]
	We have $C=\bigcup_{n\ge0} C_n$, and the filtration satisfies
	\[
	\Delta(C_n)\subseteq\sum_{i=0}^n C_i\otimes C_{n-i}.
	\]
	
	A coalgebra is \emph{pointed} if all its simple subcoalgebras are
	one-dimensional. For a pointed Hopf algebra $A$, we have
	$A_0=\K[\G(A)]$; $A$ is \emph{connected} whenever $A_0=\K 1$. The
	coradical filtration of a pointed Hopf algebra is a Hopf algebra
	filtration.
	
	Given $g,h\in\G(C)$, the space of $(g,h)$-skew-primitive elements is
	\[
	\Pp_{g,h}(C)=\{c\in C\mid \Delta(c)=c\otimes g+h\otimes c\}.
	\]
	Every element of $\Pp_{g,h}(C)$ has vanishing counit, and
	$\K(g-h)\subseteq\Pp_{g,h}(C)$. For a Hopf algebra,
	$\Pp(A)=\Pp_{1,1}(A)$ is the space of primitive elements.
	
	The first term of the coradical filtration of a pointed coalgebra
	decomposes as
	\[
	C_1=C_0+\sum_{g,h\in\G(C)}\Pp_{g,h}(C).
	\]
	Choosing vector-space complements such that
	$\Pp_{g,h}(C)=\K(g-h)\oplus Q_{g,h}$, we obtain
	$C_1=C_0\oplus\bigoplus_{g,h}Q_{g,h}$.
	
	If $D\subseteq C$ is a subcoalgebra, it inherits the filtration
	$D_n=D\cap C_n$. Coalgebra isomorphisms preserve the coradical
	filtration; consequently, Hopf algebra isomorphisms preserve
	subalgebras generated by the terms of the coradical filtration.
	
	\subsection{The diagram and the infinitesimal braiding}
	A left Yetter--Drinfeld module over $\K[G]$ is a $G$-graded vector space
	$V=\bigoplus_{s\in G}V_s$ equipped with a left $G$-action satisfying
	$t\cdot V_s\subseteq V_{tst^{-1}}$. Its braiding map is given by
	\[
	c(v\otimes w)=s\cdot w\otimes v,\qquad v\in V_s.
	\]
	
	Let $H$ be a pointed Hopf algebra and let $\pi:\gr H\to\K[\G(H)]$ denote
	the projection onto degree zero. The coinvariant subalgebra
	\[
	R=(\gr H)^{\operatorname{co}\pi}
	=\{a\in\gr H\mid (\id\otimes\pi)\Delta(a)=a\otimes 1\}
	\]
	is a connected graded braided Hopf algebra in the Yetter--Drinfeld
	category, and $\gr H\cong R\#\K[\G(H)]$. In particular,
	$\dim H=|\G(H)|\dim R$.
	
	The algebra $R$ is called the \emph{diagram} of $H$, and
	$V=R(1)$ is the \emph{infinitesimal braiding}. In particular,
	$\Pp(R)=R(1)$. The Nichols algebra $\mathcal B(V)$ is the connected
	graded braided Hopf algebra generated by $V$, with degree-one part and
	primitive part both equal to $V$ \cite{AS98b}. Equivalently, it is the
	quotient of the braided tensor algebra $T(V)$ by its largest homogeneous
	braided Hopf ideal contained in $\bigoplus_{n\ge2}T^n(V)$. The
	subalgebra of $R$ generated by $R(1)$ is isomorphic to
	$\mathcal B(R(1))$. If $R$ strictly contains $\mathcal B(R(1))$, then
	$R$ is called a \emph{non-Nichols diagram}.
	
	When $V=\K x$ is one-dimensional, its coaction reads
	$x\mapsto s\otimes x$ for some $s\in Z(G)$, and the group action is
	given by a character $\chi:G\to\K^\times$. The element $s$ is called
	the \emph{support} of $x$, and
	$c(x\otimes x)=\chi(s)x\otimes x$. If $G$ is a finite $p$-group and
	$\Char\K=p$, then every such character is trivial. In that case
	$\mathcal B(\K x)=\K[x]/(x^p)$: indeed, $x^p$ is primitive in the
	tensor algebra under trivial braiding, and the quotient possesses no
	nonzero homogeneous primitive elements in degrees $2,\dots,p-1$.

	\subsection{Coalgebra Hochschild cohomology}
	Let $C$ be a pointed coalgebra. For $g,h\in\G(C)$, let
	$M_{g,h}=\K m$ denote the one-dimensional $C$-bicomodule with
	coactions
	\[
	m\mapsto h\otimes m,\qquad m\mapsto m\otimes g.
	\]
	The coalgebra Hochschild cochain groups with coefficients in
	$M_{g,h}$ identify with tensor powers $C^{\otimes n}$,
	$C^{\otimes 0}=\K$ \cite{SO,WZZ}. The coboundary operators are
	defined by
	\begin{align*}
		d^0_{g,h}(a)&=a(g-h),\\
		d^n_{g,h}(\eta)&=-h\otimes\eta
		+\sum_{i=1}^n(-1)^{i+1}
		\bigl(\id^{\otimes(i-1)}\otimes\Delta\otimes\id^{\otimes(n-i)}\bigr)(\eta)
		+(-1)^n\eta\otimes g,\quad n\geq 1.
	\end{align*}
	Coassociativity of $\Delta$, together with the grouplike identities
	$\Delta(g)=g\otimes g$ and $\Delta(h)=h\otimes h$, implies
	$d^{n+1}_{g,h}\circ d^n_{g,h}=0$. We write $\HH^n_{g,h}(C)$ for the
	corresponding cohomology.
	
	For $n=1$,
	\[
	d^1_{g,h}(b)=\Delta(b)-b\otimes g-h\otimes b,
	\qquad
	\HH^1_{g,h}(C)=\Pp_{g,h}(C)/\K(g-h),
	\]
	where $\Pp_{g,h}(C)$ is defined as in the preceding subsection. The
	$2$-cocycle condition is
	\[
	(\Delta\otimes\id)(\omega)+\omega\otimes g
	=(\id\otimes\Delta)(\omega)+h\otimes\omega.
	\]
	
	When $C$ is finite-dimensional, set $A=C^*$. Via the natural
	identification
	\[
	C^{\otimes n}\cong \Hom_{\K}(A^{\otimes n},\K),
	\]
	the differential $d^n_{g,h}$ is the negative of the algebra
	Hochschild differential for the one-dimensional $A$-bimodule with
	left evaluation at $h$ and right evaluation at $g$. Hence
	\begin{align}\label{eq:coalgebra-ext}
		\HH^n_{g,h}(C)\cong
		\operatorname{Ext}^n_{A}(\K_g,\K_h),
	\end{align}
	see \cite[Proposition 1.4]{SO}.
	
	\begin{lem}\label{lem:prelim-extension-injection}
		Let $D\subseteq C$ be pointed coalgebras such that $C_1\subseteq D$,
		and let $g,h\in\G(C)$. Define
		\[
		E_{g,h}(C,D)=\{c\in C\mid d^1_{g,h}(c)\in D\otimes D\}.
		\]
		Then $D\subseteq E_{g,h}(C,D)$, and the map
		\[
		E_{g,h}(C,D)/D\longrightarrow\HH^2_{g,h}(D),
		\qquad c+D\mapsto [d^1_{g,h}(c)]
		\]
		is injective.
	\end{lem}
	\begin{proof}
		Since $C_1\subseteq D$, we have $g,h\in D$, and therefore
		$d^1_{g,h}(D)\subseteq D\otimes D$. For $c\in E_{g,h}(C,D)$, the
		element $d^1_{g,h}(c)$ belongs to $D\otimes D$ by definition, and
		$d^2_{g,h}d^1_{g,h}(c)=0$, so it is a $2$-cocycle of $D$. If
		$c'\in c+D$, say $c'=c+b$ with $b\in D$, then
		$d^1_{g,h}(c')-d^1_{g,h}(c)=d^1_{g,h}(b)\in B^2_{g,h}(D)$; hence
		the class $[d^1_{g,h}(c)]$ is independent of the representative.
		Thus the map is well defined.
		
		For injectivity, suppose that $[d^1_{g,h}(c)]=0$ in
		$\HH^2_{g,h}(D)$. Then there exists $b\in D$ such that
		$d^1_{g,h}(c)=d^1_{g,h}(b)$. It follows that
		$c-b\in\Pp_{g,h}(C)\subseteq C_1\subseteq D$, whence $c\in D$.
		Therefore the kernel of the map is trivial.
	\end{proof}
	
	This injection encodes the relation between the coradical filtration
	and second cohomology used to adjoin new generators; cf.\
	\cite{SO,WZZ}.
	
	If $C=\bigoplus_{j\ge0}C(j)$ is nonnegatively graded with $g,h$ in
	degree zero, total tensor degree induces an internal grading
	\[
	\HH^n_{g,h}(C)=\bigoplus_{j\ge0}\HH^{n,j}_{g,h}(C).
	\]
	Let $R$ be a connected graded coalgebra with $\Pp(R)=R(1)$, and let
	$S$ be a graded subcoalgebra agreeing with $R$ in all degrees less
	than $m\ge2$. The map
	\[
	R(m)/S(m)\longrightarrow\HH^{2,m}_{1,1}(S),\qquad
	r+S(m)\longmapsto[\bar\Delta(r)],
	\]
	where $\bar\Delta(r)=\Delta(r)-r\otimes1-1\otimes r$, is well defined
	and injective. Indeed, for $r\in R(m)$, the reduced coproduct
	$\bar\Delta(r)$ lies in
	$\sum_{i=1}^{m-1}S(i)\otimes S(m-i)$. Moreover,
	$d^2_{1,1}(\bar\Delta(r))=d^2_{1,1}d^1_{1,1}(r)=0$, so
	$\bar\Delta(r)$ is a $2$-cocycle of $S$. If $r'=r+s$ with
	$s\in S(m)$, then $\bar\Delta(r')-\bar\Delta(r)=d^1_{1,1}(s)$, and
	hence $[\bar\Delta(r')]=[\bar\Delta(r)]$. Thus the map is well
	defined.
	
	If $[\bar\Delta(r)]=0$, then $\bar\Delta(r)=d^1_{1,1}(b)$ for some
	$b\in S(m)$. Hence $\bar\Delta(r-b)=0$, so $r-b$ is a primitive
	element of degree $m\ge2$. Since $\Pp(R)=R(1)$, it follows that
	$r-b=0$, and therefore $r\in S(m)$. Thus the map is injective. The
	same degree argument applies to connected graded braided Hopf
	algebras in a Yetter--Drinfeld category; cf.\ \cite{NW}.
	
	For finite-dimensional pointed Hopf algebras, filtering the cochain
	complex by the coradical filtration yields the dimension bound
	\[
	\dim\HH^n_{g,h}(H)\leq\dim\HH^n_{g,h}(\gr H)
	\]
	from \cite[Lemma 8.4]{WZZ}.

	\section{Simple extensions}\label{sec:fixed-base}
	
	Throughout this section, $L$ is a finite-dimensional pointed Hopf algebra
	over $\K$, and $L^+=\ker\epsilon_L$. 
	We identify $L$ with its image under each specified Hopf embedding.
	
	\begin{defn}\label{def:index}
		Let $L$ be a Hopf subalgebra of a finite-dimensional Hopf algebra $H$.
		By the Nichols--Zoeller theorem, $H$ is a free left $L$-module. For a
		positive integer $n$, we say that $H$ has \emph{index} $n$ over $L$ if
		$H$ is free of rank $n$ as a left $L$-module. Equivalently,
		$\dim H=n\dim L$.
	\end{defn}
	
	\begin{defn}\label{def:intrinsic}
		A Hopf subalgebra $L\subseteq H$ is called \emph{intrinsic} if it is
		preserved by every Hopf automorphism of $H$.
	\end{defn}
	
	\begin{defn}\label{def:sge}
		A \emph{simple extension} of $L$ is a
		finite-dimensional pointed Hopf algebra $H$ equipped with a Hopf
		embedding $L\hookrightarrow H$ such that $H$ is generated as an algebra
		by $L$ and one element of $H\setminus L$.
		A \emph{base-preserving isomorphism} is a Hopf algebra isomorphism
		$\Phi:H\to H'$ satisfying $\Phi(L)=L$.
	\end{defn}
	
	\begin{rmk}
		\normalfont
		The subspace $L^+H$ is a right ideal and a coideal, so $H/L^+H$ is
		a quotient coalgebra and a right $H$-module. When $L$ is normal in $H$,
		one has $L^+H=HL^+$, and this subspace is a Hopf ideal.
		In that case $H/L^+H$ is a quotient Hopf algebra.
	\end{rmk}
	
	\begin{defn}\label{def:normalized}
		Fix $g,h\in \G(L)$. A \emph{normalized quadratic extension with specified support
			$(g,h)$} is a simple extension admitting an element
		$w\in\ker\epsilon_H$ such that $\{1,w\}$ is a free left $L$-basis of $H$
		and
		\begin{align}
			\Delta(w)&=w\otimes g+h\otimes w+\omega,
			\qquad \omega\in L\otimes L,\label{eq:coproduct}\\
			wa&=\sigma(a)w+\delta(a)\quad(a\in L),\label{eq:ore}\\
			w^2&=u+vw,\qquad u,v\in L.\label{eq:power}
		\end{align}
		Here $\sigma:L\to L$ and $\delta:L\to L$ are the uniquely determined
		linear maps defined by \eqref{eq:ore}. An isomorphism of extensions
		with specified support is a base-preserving Hopf algebra isomorphism
		whose restriction to $L$ fixes both $g$ and $h$.
	\end{defn}
	
	By associativity of $H$ and uniqueness of the left $L$-basis expansion,
	the maps $\sigma$ and $\delta$ necessarily satisfy
	\begin{equation}\label{eq:sigma-delta}
		\sigma(1)=1,\qquad \delta(1)=0,\qquad
		\sigma(ab)=\sigma(a)\sigma(b),\qquad
		\delta(ab)=\sigma(a)\delta(b)+\delta(a)b
	\end{equation}
	for all $a,b\in L$. In particular, $\sigma$ is a unital algebra
	endomorphism and $\delta$ is a $(\sigma,1)$-derivation.
	
	\begin{pro}\label{pro:cocycle}
		In a normalized quadratic extension, $\omega\in L^+\otimes L^+$ and
		\begin{equation}\label{eq:cocycle-cond}
			(\Delta\otimes\id)(\omega)+\omega\otimes g
			=(\id\otimes\Delta)(\omega)+h\otimes\omega.
		\end{equation}
		Moreover, $\epsilon\delta=0$ and $\epsilon(u)=0$.
	\end{pro}
	\begin{proof}
		Applying $\epsilon\otimes\id$ and $\id\otimes\epsilon$ to
		\eqref{eq:coproduct} gives
		$(\epsilon\otimes\id)(\omega)=(\id\otimes\epsilon)(\omega)=0$.
		Since $L=\K 1\oplus L^+$, $\omega\in L^+\otimes L^+$. Coassociativity on $w$ gives
		\eqref{eq:cocycle-cond} after cancellation of the terms containing $w$.
		Applying $\epsilon$ to \eqref{eq:ore} and \eqref{eq:power}, and using
		$\epsilon(w)=0$, yields $\epsilon\delta(a)=0$ for all $a\in L$ and
		$\epsilon(u)=0$.
	\end{proof}
	
	\begin{defn}\label{def:cohomology}
		For the one-dimensional $L$-bicomodule $M_{g,h}$ with left coaction
		$m\mapsto h\otimes m$ and right coaction $m\mapsto m\otimes g$,
		the \emph{twisted coalgebra Hochschild complex} has cochain spaces
		$C^n=L^{\otimes n}$, with $C^0=\K$. With the sign convention of
		Section~\ref{sec:prelim}, its first two differentials are
		\begin{align*}
			d^1_{g,h}(b)&=\Delta(b)-b\otimes g-h\otimes b,\\
			d^2_{g,h}(\eta)&=(\Delta\otimes\id)(\eta)+\eta\otimes g
			-(\id\otimes\Delta)(\eta)-h\otimes\eta.
		\end{align*}
		Write
		\[
		Z^2_{g,h}(L)=\ker d^2_{g,h},\qquad
		B^2_{g,h}(L)=d^1_{g,h}(L),\qquad
		\HH^2_{g,h}(L)=Z^2_{g,h}(L)/B^2_{g,h}(L).
		\]
		A \emph{normalized cocycle} is an element of
		$Z^2_{g,h}(L)\cap(L^+\otimes L^+)$.
	\end{defn}
	
	\begin{rmk}
		\normalfont
		Coassociativity gives $d^2_{g,h}d^1_{g,h}=0$. If two normalized cocycles
		differ by $d^1_{g,h}(b)$, then $b\in L^+$, since
		\[
		(\epsilon\otimes\id)d^1_{g,h}(b)=-\epsilon(b)g.
		\]
		Consequently, adding a coboundary to a normalized coproduct correction
		is precisely the effect of replacing $w$ by $w+b$, with $b\in L^+$.
	\end{rmk}

	\subsection{Normalized quadratic  extensions and Ore-type relations}
	
	\begin{pro}\label{pro:associative-data}
		Assume that $\sigma$ and $\delta$ satisfy \eqref{eq:sigma-delta}.
		Then the algebra presented by \eqref{eq:ore}--\eqref{eq:power} has free  left $L$-basis $\{1,w\}$ if and only if 
		\begin{align}
			\sigma^2(a)u+\delta^2(a)&=ua+v\delta(a),\label{eq:assoc-wwa-const}\\
			\sigma^2(a)v+\delta(\sigma(a))+\sigma(\delta(a))
			&=v\sigma(a)\quad(a\in L),\label{eq:assoc-wwa-coef}\\
			\delta(u)+\sigma(v)u&=vu,\label{eq:assoc-w2-const}\\
			\sigma(u)+\sigma(v)v+\delta(v)&=u+v^2.\label{eq:assoc-w2-coef}
		\end{align}
		Here $\sigma^2$ and $\delta^2$ denote composition.
	\end{pro}
	\begin{proof}
		Since $\sigma$ is a unital algebra endomorphism and $\delta$ is a
		$(\sigma,1)$-derivation, the overlap of type $wab$ is already
		resolved. Thus
		associativity reduces to the overlaps $wwa$ and $www$.
		
		For $wwa$, reducing $w(wa)$ gives
		\[
		w(\sigma(a)w+\delta(a))
		=\sigma^2(a)w^2+\delta(\sigma(a))w+\sigma(\delta(a))w+\delta^2(a).
		\]
		Using $w^2=u+vw$, this becomes
		\[
		\sigma^2(a)u+\delta^2(a)
		+\bigl(\sigma^2(a)v+\delta(\sigma(a))+\sigma(\delta(a))\bigr)w.
		\]
		On the other hand,
		\[
		(w^2)a=(u+vw)a=ua+v(\sigma(a)w+\delta(a))
		=ua+v\delta(a)+v\sigma(a)w.
		\]
		Comparing constant terms and coefficients of $w$ yields
		\eqref{eq:assoc-wwa-const} and \eqref{eq:assoc-wwa-coef}.
		
		Similarly, from $w(w^2)=w(u+vw)$ we obtain
		\[
		\delta(u)+\sigma(v)u
		+\bigl(\sigma(u)+\sigma(v)v+\delta(v)\bigr)w,
		\]
		while
		\[
		(w^2)w=(u+vw)w
		=uw+vw^2
		=vu+(u+v^2)w.
		\]
		Comparing the two expressions gives
		\eqref{eq:assoc-w2-const} and \eqref{eq:assoc-w2-coef}.
		
		Conversely, assume that the four identities
		\eqref{eq:assoc-wwa-const}--\eqref{eq:assoc-w2-coef} hold. Choose a
		basis of $L$ containing $1$, and present $L$ by the multiplication
		table of the remaining basis elements. Assign weight $1$ to each base
		letter and weight $2$ to $w$, and order monomials first by weighted
		degree and then lexicographically, with $w$ larger than every base
		letter. The reduction rules
		\[
		wa\mapsto\sigma(a)w+\delta(a),\qquad
		w^2\mapsto u+vw
		\]
		strictly decrease this order. Overlaps coming from the multiplication
		table of $L$ are resolved by associativity of $L$. The only remaining
		overlaps are of the forms $wab$, $wwa$, and $www$. The first is
		resolved by \eqref{eq:sigma-delta} and the
		others by the identities \eqref{eq:assoc-wwa-const}--
		\eqref{eq:assoc-w2-coef}. Therefore the Diamond Lemma \cite{B}
		applies and yields the free left $L$-basis $\{1,w\}$.
	\end{proof}

	For data satisfying Proposition~\ref{pro:associative-data}, let $A$
	be the associative algebra with free left $L$-basis $\{1,w\}$ and
	relations
	\[
	wa=\sigma(a)w+\delta(a)\quad(a\in L),\qquad w^2=u+vw.
	\]
	Set
	\[
	W=w\otimes g+h\otimes w+\omega\in A\otimes A.
	\]
	We now determine when the comultiplication of $L$ extends to an
	algebra map $\Delta_A:A\to A\otimes A$ with $\Delta_A(w)=W$.
	Preservation of the Ore relation by $\Delta_A$ is equivalent to
	\begin{equation}\label{eq:comp-full}
		\Delta(\sigma(a))W+\Delta(\delta(a))=W\Delta(a)
		\qquad(a\in L).
	\end{equation}
	Comparing coefficients in the free left $L\otimes L$-basis
	$\{1\otimes1,w\otimes1,1\otimes w,w\otimes w\}$ gives
	\begin{align}
		\Delta(\sigma(a))(1\otimes g)
		&=\sum\sigma(a_{(1)})\otimes ga_{(2)},\label{eq:comp1}\\
		\Delta(\sigma(a))(h\otimes1)
		&=\sum ha_{(1)}\otimes\sigma(a_{(2)}),\label{eq:comp2}\\
		\Delta(\sigma(a))\omega+\Delta(\delta(a))
		&=\omega\Delta(a)+\sum\delta(a_{(1)})\otimes ga_{(2)}
		+\sum ha_{(1)}\otimes\delta(a_{(2)}).\label{eq:comp3}
	\end{align}
	Preservation of the square relation by $\Delta_A$ is equivalent to
	the identity
	\begin{equation}\label{eq:comp4}
		W^2=\Delta(u)+\Delta(v)W
	\end{equation}
	in $A\otimes A$.
	
	\begin{defn}\label{def:extension-data}
		For $g,h\in G(L)$, let $\mathcal D(L;g,h)$ be the set of tuples
		$\Theta=(g,h,\omega,\sigma,\delta,u,v)$ satisfying the following conditions:
		\begin{enumerate}
			\item $\omega\in Z^2_{g,h}(L)\cap(L^+\otimes L^+)$;
			\item $\sigma$ is a unital algebra endomorphism, $\delta$ is a
			$(\sigma,1)$-derivation, and $\epsilon\delta=0$;
			\item $u\in L^+$, $v\in L$, and the four identities of
			Proposition~\ref{pro:associative-data} hold;
			\item Equations \eqref{eq:comp1}--\eqref{eq:comp3} hold for every
			$a\in L$, and \eqref{eq:comp4} holds in the algebra $A\otimes A$
			constructed from these data.
		\end{enumerate}
	\end{defn}
	
	\begin{pro}\label{pro:reconstruction}
		Every $\Theta\in\mathcal D(L;g,h)$ determines a normalized Hopf
		extension $H_\Theta$ of dimension $2\dim L$, with coradical $L_0$.
		Conversely, every normalized extension determines a tuple in
		$\mathcal D(L;g,h)$.
	\end{pro}
	\begin{proof}
		Proposition~\ref{pro:associative-data} gives an algebra $A$ with free
		left $L$-basis $\{1,w\}$. Equations \eqref{eq:comp-full} and
		\eqref{eq:comp4} extend $\Delta_L$ to an algebra map
		$\Delta_A:A\to A\otimes A$ with $\Delta_A(w)=W$.
		The cocycle identity gives coassociativity on $w$, and hence on $A$.
		The conditions $\epsilon\delta=0$ and $\epsilon(u)=0$ extend
		$\epsilon_L$ to an algebra map $\epsilon_A$ with $\epsilon_A(w)=0$.
		The normalization of $\omega$ gives both counit identities.
		
		For $a\in L$, the identity $\Delta_A(aw)=\Delta_L(a)W$ implies
		\[
		\Delta_A(A)\subseteq A\otimes L+L\otimes A.
		\]
		Hence $F_0=L$ and $F_n=A$ for $n\geq1$ form an exhaustive coalgebra
		filtration. Thus $A_0\subseteq L$ and $L_0=L\cap A_0$, which gives $A_0=L_0$.
		Since $A_0=\K[\G(L)]$ is a Hopf subalgebra, it follows by \cite[Proposition 7.6.3]{R11} that
		$A$ is a pointed Hopf algebra.
		This is $H_\Theta$.
		
		For the converse, associativity and the free left $L$-basis give
		the identities of Proposition~\ref{pro:associative-data}.
		Proposition~\ref{pro:cocycle} gives normalization and the counit
		conditions, while preservation of the defining relations by
		$\Delta_H$ gives \eqref{eq:comp1}--\eqref{eq:comp4}.
	\end{proof}
	
	\begin{rmk}
		The compatibility equations force $\sigma$ to be an automorphism.
		Indeed, $\chi=\epsilon\sigma:L\to\K$ is a character, and applying
		the counit to \eqref{eq:comp1} and \eqref{eq:comp2} gives
		\[
		\sigma(a)=\sum\chi(a_{(1)})\,ga_{(2)}g^{-1}
		=\sum ha_{(1)}h^{-1}\chi(a_{(2)}).
		\]
		Thus $\sigma$ is a winding automorphism composed with conjugation
		by a group-like element. 
	\end{rmk}
	
	\subsection{Gauge transformations and equivalence of data}
	
	\begin{defn}\label{def:gauge}
		The \emph{gauge group} $\Gau_{g,h}(L)=\Gau(L)$ consists of pairs
		$(\alpha,b)\in\K^\times\times L^+$, with multiplication
		\[
		(\alpha,b)(\alpha',b')=(\alpha\alpha',\alpha b'+b).
		\]
		The pair $(\alpha,b)$ specifies the change of generator
		$w'=\alpha w+b$.
	\end{defn}
	
	In the generator $w'$, the data are $T_{\alpha,b}\Theta$, where
	\begin{align*}
		\omega'&=\alpha\omega+d^1_{g,h}(b),&\sigma'&=\sigma,\\
		\delta'(a)&=\alpha\delta(a)+ba-\sigma(a)b,\\
		v'&=\alpha v+\sigma(b)+b,\\
		u'&=\alpha^2u+\alpha\delta(b)+b^2-v'b.
	\end{align*}
	The support is $(g,h)$. For $\psi\in\Aut_{\mathrm{Hopf}}(L)$, transport
	of data is defined by
	\[
	P_\psi\Theta=
	\bigl(\psi(g),\psi(h),(\psi\otimes\psi)\omega,
	\psi\sigma\psi^{-1},\psi\delta\psi^{-1},\psi(u),\psi(v)\bigr).
	\]
	Write
	\[
	B_{g,h}=\Aut_{\mathrm{Hopf}}(L)_{g,h}
	=\{\psi\in\Aut_{\mathrm{Hopf}}(L):\psi(g)=g,\ \psi(h)=h\}.
	\]
	
	\begin{lem}\label{lem:elementary-iso}
		The maps $T_{\alpha,b}$ and $P_\psi$ preserve the data sets with their
		respective supports and satisfy
		\[
		T_{\alpha,b}T_{\alpha',b'}=T_{\alpha\alpha',\alpha b'+b},
		\qquad
		P_\psi T_{\alpha,b}P_\psi^{-1}=T_{\alpha,\psi(b)}.
		\]
		Each map induces a base-preserving Hopf algebra isomorphism between
		the corresponding reconstructed extensions.
	\end{lem}
	\begin{proof}
		The elements $1,\alpha w+b$ form a free left $L$-basis of $H_\Theta$.
		Substitution gives the displayed formulas for the transformed data,
		so Proposition~\ref{pro:reconstruction} shows that these data belong
		to $\mathcal D(L;g,h)$. The map
		$H_{T_{\alpha,b}\Theta}\to H_\Theta$ fixing $L$ and sending its
		generator to $\alpha w+b$ is a Hopf algebra isomorphism.
		Transport gives a Hopf algebra isomorphism
		$H_\Theta\to H_{P_\psi\Theta}$ restricting to $\psi$ on $L$ and
		sending $w$ to the target generator. Direct substitution in the
		transformation formulas gives the two composition identities.
	\end{proof}
	
	\begin{defn}\label{def:equivalence}
		Let $\Gamma_{g,h}$ be the semidirect product with elements
		$(\alpha,b,\psi)\in\K^\times\times L^+\times B_{g,h}$ and multiplication
		\[
		(\alpha,b,\psi)(\alpha',b',\psi')
		=(\alpha\alpha',b+\alpha\psi(b'),\psi\psi').
		\]
		Its action on $\mathcal D(L;g,h)$ is
		\[
		(\alpha,b,\psi)\cdot\Theta=T_{\alpha,b}P_\psi\Theta.
		\]
		Two tuples are \emph{equivalent} if they belong to the same orbit.
	\end{defn}
	
	\subsection{Orbit classification theorem}
	
	\begin{lem}\label{lem:generator-rigidity}
		Let $H$ and $H'$ be normalized extensions. Denote their generators
		by $w$ and $w'$ and their supports by $(g,h)$ and $(g',h')$,
		respectively. If $\Phi:H\to H'$ is a base-preserving
		Hopf algebra isomorphism and $\psi=\Phi|_L$, then
		\[
		\Phi(w)=\alpha k w'+b,
		\qquad \alpha\in\K^\times,\quad b\in L^+,\quad k\in \G(L),
		\]
		where $\psi(g)=kg'$ and $\psi(h)=kh'$.
		In particular, $k=1$ when the supports agree and $\psi$ fixes them.
	\end{lem}
	\begin{proof}
		Write $\Phi(w)=aw'+b$ with $a,b\in L$. Comparing the
		$w'\otimes1$ and $1\otimes w'$ coefficients in the coproduct gives
		\[
		\Delta(a)(1\otimes g')=a\otimes\psi(g),\qquad
		\Delta(a)(h'\otimes1)=\psi(h)\otimes a.
		\]
		Put $k=\psi(g)(g')^{-1}\in \G(L)$. The first equality gives
		$\Delta(a)=a\otimes k$, hence $a=\epsilon(a)k$.
		Surjectivity of $\Phi$ gives $a\ne0$, so $\alpha=\epsilon(a)\ne0$.
		Substitution in the second equality gives $kh'=\psi(h)$.
		Applying the counit to $\Phi(w)$ gives $\epsilon(b)=0$.
	\end{proof}
	
	\begin{thm}\label{thm:orbit}
		Fix $g,h\in \G(L)$. Reconstruction induces a bijection between
		$\mathcal D(L;g,h)/\Gamma_{g,h}$ and the isomorphism classes of
		normalized extensions with specified support $(g,h)$, with
		isomorphisms as in Definition~\ref{def:normalized}.
	\end{thm}
	\begin{proof}
		Proposition~\ref{pro:reconstruction} assigns an extension to each
		tuple and a tuple to each normalized extension. For two normalized
		generators of the same extension with the same support,
		Lemma~\ref{lem:generator-rigidity}, applied to the identity map,
		gives a change $w\mapsto\alpha w+b$, with $b\in L^+$.
		Their tuples therefore belong to the same gauge orbit.
		
		Every element of $\Gamma_{g,h}$ induces an isomorphism preserving
		the specified support by Lemma~\ref{lem:elementary-iso}.
		Conversely, let $\Phi:H_\Theta\to H_{\Theta'}$ be such an
		isomorphism, and put $\psi=\Phi|_L\in B_{g,h}$.
		Lemma~\ref{lem:generator-rigidity} gives
		$\Phi(w)=\alpha w'+b$. Transport of the source data by $\psi$
		gives the data of this generator in the target, so
		\[
		P_\psi\Theta=T_{\alpha,b}\Theta'.
		\]
		Thus $\Theta$ and $\Theta'$ belong to the same $\Gamma_{g,h}$-orbit.
	\end{proof}
	
	\begin{rmk}\label{rmk:unmarked-supports}
		\normalfont
		For extensions with varying support, set
		\[
		\mathcal D(L)=\bigsqcup_{g,h\in G(L)}\mathcal D(L;g,h).
		\]
		Base-preserving isomorphism classes correspond to the equivalence
		relation generated by gauge transformations, all transports $P_\psi$,
		and changes $w\mapsto kw$ with $k\in \G(L)$.
		The last change sends $(g,h)$ to $(kg,kh)$ and gives
		\begin{align*}
			\omega'&=(k\otimes k)\omega,&
			\sigma'(a)&=k\sigma(a)k^{-1},&\delta'(a)&=k\delta(a),\\
			u'&=k\sigma(k)u,&
			v'&=\bigl(k\sigma(k)v+k\delta(k)\bigr)k^{-1}.
		\end{align*}
		These changes preserve normalization, and
		Lemma~\ref{lem:generator-rigidity} supplies every base-preserving
		Hopf algebra isomorphism in this way.
	\end{rmk}

	\subsection{Classification of quadratic extensions with the same coradical}
	\begin{lem}\label{lem:adapted-rank-two-basis}
		Let $L\subseteq H$ be finite-dimensional pointed Hopf algebras such that
		$\G(H)=\G(L)$ and $\dim H=2\dim L$. Then there exist
		$x\in H\setminus L$ and $g\in\G(L)$ such that
		\(
		H=L\oplus Lx
		\)
		and
		\(
		\Delta(x)=x\otimes g+1\otimes x+\omega
		\)
		for some $\omega\in L\otimes L$.
	\end{lem}
	
	\begin{proof}
		Set
		\[
		C^{(0)}=L,\qquad
		C^{(n+1)}=L\wedge C^{(n)}.
		\]
		Since $H$ and $L$ are pointed and $\G(H)=\G(L)$, we have
		$H_0=L_0\subseteq L$. Hence by induction, 
		\(
		H_n\subseteq C^{(n)}
		\)
		for every $n\ge0$. Since the filtration $\{H_n\}$ is
		exhaustive, so is the filtration $\{C^{(n)}\}$.
		
		Each $C^{(n)}$ is a left $L$-submodule of $H$. Moreover,
		$C^{(n+1)}/C^{(n)}$ is a left $L$-Hopf module: the left
		$L$-action is induced by multiplication, and the left coaction is
		induced by the coproduct. Hence, by the fundamental theorem of Hopf
		modules, each nonzero quotient
		$C^{(n+1)}/C^{(n)}$ is a free left $L$-module.
		
		Since $H\ne L$ and the filtration is exhaustive, we must have
		$C^{(1)}=L\wedge L\ne L$; otherwise
		$C^{(n)}=L$ for every $n$. Thus $C^{(1)}/L\ne0$, which is a free left $L$-module. Therefore
		\[
		\dim(C^{(1)}/L)\ge\dim L.
		\]
		On the other hand,
		\[
		C^{(1)}\subseteq H,\qquad \dim H=2\dim L,
		\]
		so equality must hold and consequently $C^{(1)}=H.$ Therefore $	H=L\wedge L$ and $M:=H/L$
		is a free left $L$-module of rank one.
		
		The inclusion $H=L\wedge L$ gives $M$ natural left and right
		$L$-comodule structures
		\[
		\lambda:M\longrightarrow L\otimes M,
		\qquad
		\rho:M\longrightarrow M\otimes L,
		\]
		induced by $\Delta$. Together with the left $L$-action,
		$(M,\lambda)$ is a left $L$-Hopf module. Hence
		\[
		M\cong L\otimes M^{\operatorname{co}L}.
		\]
		Since $M$ has left $L$-rank one,
		$\dim M^{\operatorname{co}L}=1$. Choose
		$0\ne\bar x\in M^{\operatorname{co}L}$. Then
		$\bar x$ is a free left $L$-basis of $M$ and $	\lambda(\bar x)=1\otimes\bar x.$
		The left and right coactions commute, so the one-dimensional space
		$M^{\operatorname{co}L}$ is stable under the right coaction. Hence
		there exists $g\in\G(L)$ such that
		\[
		\rho(\bar x)=\bar x\otimes g.
		\]
		Choose a lift $x\in H$ of $\bar x$. Since $\bar x$ generates
		$H/L$ freely as a left $L$-module,
		\[
		H=L\oplus Lx.
		\]
		The two coaction identities imply
		\[
		\Delta(x)-1\otimes x\in H\otimes L,
		\qquad
		\Delta(x)-x\otimes g\in L\otimes H.
		\]
		Therefore
		\[
		\Delta(x)-x\otimes g-1\otimes x
		\in(H\otimes L)\cap(L\otimes H)
		=L\otimes L.
		\]
		This proves the result.
	\end{proof}

	\begin{thm}[Classification of index-two extensions with the same coradical]
		\label{thm:classification-index-two}
		Let $L$ be a finite-dimensional pointed Hopf algebra over $\K$, and
		let $H$ be a finite-dimensional pointed Hopf algebra containing $L$
		as a Hopf subalgebra. Then the following are equivalent:
		\begin{enumerate}
			\item[(i)] $H$ is a normalized quadratic extension of $L$ in the
			sense of Definition~\ref{def:normalized};
			\item[(ii)] $\dim H=2\dim L$ and $H_0=L_0$.
		\end{enumerate}
		Consequently, the base-preserving isomorphism classes of all such
		algebras $H$ are parameterized by the orbit spaces
		\[
		\bigsqcup_{g,h\in G(L)}\mathcal D(L;g,h)\Big/\sim,
		\]
		where $\sim$ is the equivalence relation generated by the gauge
		transformations, the transports $P_\psi$ for
		$\psi\in\Aut_{\mathrm{Hopf}}(L)$, and the changes $w\mapsto kw$
		for $k\in G(L)$, as in Remark~\ref{rmk:unmarked-supports}.
	\end{thm}
	\begin{proof}
		We first show that (i) implies (ii). If $H$ is a normalized
		quadratic extension of $L$, then by definition there exists
		$w\in\ker\epsilon_H$ such that $\{1,w\}$ is a free left $L$-basis
		of $H$. Hence
		\[
		H=L\oplus Lw,
		\]
		and therefore $\dim H=2\dim L$.
		
		It remains to prove $H_0=L_0$. Since $L$ is a Hopf subalgebra,
		$L_0\subseteq H_0$. Conversely, let $x\in \G(H)$. Write
		\[
		x=a+bw
		\]
		with $a,b\in L$. Because $x$ is group-like,
		\[
		\Delta(x)=x\otimes x=(a+bw)\otimes(a+bw).
		\]
		On the other hand, using
		\[
		\Delta(w)=w\otimes g+h\otimes w+\omega
		\]
		with $g,h\in \G(L)$ and $\omega\in L\otimes L$, we obtain
		\[
		\Delta(x)=\Delta(a)+\Delta(b)(w\otimes g+h\otimes w+\omega).
		\]
		Since $\Delta(a),\Delta(b)\in L\otimes L$, the right-hand side
		contains no term of the form $w\otimes w$. Comparing the
		$w\otimes w$-component in $(a+bw)\otimes(a+bw)$ gives
		$b\otimes b=0$, hence $b=0$. Thus $x=a\in L$. Therefore
		$\G(H)\subseteq \G(L)$, and so $H_0=L_0$.
		
		Now assume (ii). Since $H_0=L_0$, we have $\G(H)=\G(L)$. By
		Lemma~\ref{lem:adapted-rank-two-basis}, there exist
		$x\in H\setminus L$ and $g\in\G(L)$ such that
		\[
		H=L\oplus Lx,\qquad 	\Delta(x)=x\otimes g+1\otimes x+\omega,
		\]
		for some $\omega\in L\otimes L$.
		
		Replacing $x$ by $x-\epsilon(x)1$ does not change the direct sum
		decomposition. Thus we may assume $\epsilon(x)=0$; after this
		replacement the coproduct still has the form
		\[
		\Delta(x)=x\otimes g+1\otimes x+\omega
		\]
		with $\omega\in L\otimes L$. Applying
		$\epsilon\otimes\id$ and $\id\otimes\epsilon$ gives
		\[
		\omega\in L^+\otimes L^+.
		\]
		
		Because $\{1,x\}$ is a free left $L$-basis, for every $a\in L$
		there are unique elements $\sigma(a),\delta(a)\in L$ such that
		\[
		xa=\sigma(a)x+\delta(a).
		\]
		Associativity of $H$ implies that $\sigma$ is a unital algebra
		endomorphism of $L$ and that $\delta$ is a $(\sigma,1)$-derivation.
		Writing $x^2$ in the same basis gives unique $u,v\in L$
		with $	x^2=u+vx$. Then it follows by $\epsilon(x)=0$ that $\epsilon(u)=0$, that is, $u\in L^+$. Hence $H$ is a normalized quadratic extension of
		$L$, with specified support $(g,1)$.
		
		The final classification statement follows by combining the
		equivalence of (i) and (ii) with Theorem~\ref{thm:orbit} and
		Remark~\ref{rmk:unmarked-supports}.
	\end{proof}
	
	\begin{pro}\label{pro:normal-quadratic-extension-char2}
		Let $L$ be a finite-dimensional pointed Hopf algebra over $\K$, and
		let $H$ be a normalized quadratic extension of $L$. If $L$ is normal
		in $H$, then $\Char\K=2$.
	\end{pro}
	\begin{proof}
		Since $L$ is normal in $H$, the subspace $L^+H=HL^+$ is a Hopf
		ideal of $H$, so the quotient $	Q:=H/L^+H$ is a Hopf algebra. By $H=L\oplus Lw$ as a left $L$-module, we have
		\[
		L^+H=L^+\oplus L^+w,
		\]
		hence
		\[
		Q\cong \K 1\oplus \K x,
		\]
		where $x$ denotes the image of $w$. In particular, $\dim Q=2$ and
		$x\neq0$.
		
		We compute the coproduct of $x$ in $Q$. In $H$,
		\[
		\Delta(w)=w\otimes g+h\otimes w+\omega
		\]
		with $g,h\in\G(L)$ and $\omega\in L^+\otimes L^+$. Passing to $Q$,
		the images of $g$ and $h$ are $1$, since $g-1,h-1\in L^+$, and the
		image of $\omega$ is $0$, since $\omega\in L^+\otimes L^+$. Hence
		\[
		\Delta_Q(x)=x\otimes1+1\otimes x,
		\]
		that is, $x$ is a primitive element of $Q$.
		
		Write $x^2=\lambda x+\mu 1$ with $\lambda,\mu\in\K$. Applying the
		counit to this identity and using $\epsilon(x)=0$, we get
		$\mu=0$, that is, $	x^2=\lambda x.$ Applying $\Delta_Q$ to this relation, the left-hand side is
		\[
		\Delta_Q(x^2)=\Delta_Q(x)^2
		=(x\otimes1+1\otimes x)^2
		=x^2\otimes1+2\,x\otimes x+1\otimes x^2,
		\]
		while the right-hand side is
		\[
		\Delta_Q(\lambda x)=\lambda x\otimes1+\lambda 1\otimes x.
		\]
		Substituting $x^2=\lambda x$ and comparing the middle terms yields
		$2\,x\otimes x=0$. Since $x\otimes x\neq0$,  $2=0$ and hence
		$\Char\K=2$.
	\end{proof}
	\begin{rmk}
		Proposition~\ref{pro:normal-quadratic-extension-char2} shows that
		a normalized quadratic extension of a finite-dimensional pointed
		Hopf algebra cannot have a normal base when
		$\Char\K\neq2$. In characteristic $2$, this obstruction vanishes,
		which allows the nontrivial quadratic extensions studied in
		Sections~\ref{sec:case-dimR4} and~\ref{sec:case-dimR8}.
	\end{rmk}
	
	\section{The classification in dimension $16$}
	\label{sec:classification}
	
	Throughout this section, $\K$ is algebraically closed of characteristic
	$2$. Put $\Z_4=\langle g\mid g^4=1\rangle$ and
	$\Z_2\times\Z_2=\langle g,h\mid g^2=h^2=1,\ gh=hg\rangle$.
	For $\Z_4$, let $\beta_1(g)=1$; for $\Z_2\times\Z_2$, let
	$\beta_0(g)=1$, $\beta_0(h)=0$, and
	$\beta_t(g)=1$, $\beta_t(h)=t$.
	These are homomorphisms to $(\K,+)$.
	
	In the tables below, the symbols $P_G(\epsilon,\beta,\mu)$,
	$\mathcal A_a$, $C(r,\lambda)$, $D(a,b,\lambda)$,
	$H_0(\epsilon,\beta,\gamma,\lambda,\mu,\nu)$, and
	$H_1(\epsilon,\kappa,\tau)$ denote the Hopf algebras whose defining
	relations are given in Proposition~\ref{pro:primitive-support-data}
	and Lemmas~\ref{lem:case-m=1}, \ref{lem:case-m=2theta},
	\ref{lem:42n-Classification-2}, \ref{lem:m0-dim16-relations}, and
	\ref{lem:m1-dim16-relations}. Each of these algebras is generated by
	the group $G$ and the indicated elements, with the common coalgebra
	structure described below. The tables list the group $G$ and support $s$ explicitly; the
	defining relations are given in the cited lemmas.
	
	Each entry in the tables denotes the entire Hopf algebra with the
	indicated presentation and support $s$. Its coalgebra structure is
	\begin{gather*}
		\Delta(t)=t\otimes t\quad(t\in G),\qquad
		\Delta(x)=x\otimes1+s\otimes x,\\
		\Delta(y)=y\otimes1+s^2\otimes y+xs\otimes x.
	\end{gather*}
	For $G=\Z_2=\langle g\rangle$, there is also a generator $z$ with
	\[
	\Delta(z)=z\otimes1+1\otimes z
	+xys\otimes x+xs\otimes xy+y\otimes y.
	\]
	The counit is $1$ on group elements and zero on $x,y,z$.
	
	Let $R$ be the diagram of $H$. 
	Since $H$ is non-connected and the infinitesimal braiding is
	one-dimensional, the possible values are $\dim R=2,4,8$. If
	$\dim R=2$, then $R=\mathcal B(R(1))$; hence $H$ is generated by
	group-like and skew-primitive elements, and this case is covered by
	\cite{X23}. Therefore the remaining non-Nichols diagrams have
	$\dim R=4$ or $\dim R=8$.  Moreover, in each of these cases we will
	construct an intrinsic eight-dimensional Hopf subalgebra $L$ with
	$L_0=H_0$, so that Theorem~\ref{thm:classification-index-two}
	applies and $H$ is a normalized quadratic extension of $L$.
	
	\begin{thm}\label{thm:dim16-classification}
		Let $H$ be a non-connected pointed Hopf algebra of dimension $16$
		over $\K$, with one-dimensional infinitesimal braiding and a diagram
		that is not a Nichols algebra. Then $H$ is isomorphic to an entry in
		Table~\ref{tab:isolated-classes} or Table~\ref{tab:parameter-families}.
		Every entry satisfies these hypotheses. Entries in distinct rows of
		either table, or in different tables, are non-isomorphic. The discrete
		parameters within Table~\ref{tab:isolated-classes} give pairwise
		non-isomorphic algebras. Within each row of
		Table~\ref{tab:parameter-families}, two entries are isomorphic exactly
		under the equivalence specified in its last column.
		Thus the list consists of $24$ individual representatives and eight
		parameter families, seven with one parameter and one with two parameters.
	\end{thm}
	
	\begin{xltabular}{\textwidth}{ccX}
		\toprule
		$G$ & $s$ & Representatives \\
		\midrule
		\endhead
		$\Z_4$ & $1$ & $P_{\Z_4}(0,0,0)$, $P_{\Z_4}(0,0,1)$,
		$P_{\Z_4}(1,0,0)$, $P_{\Z_4}(1,\beta_1,0)$ \\[2mm]
		$\Z_4$ & $g$ & $\mathcal A_0$, $\mathcal A_1$ \\[2mm]
		$\Z_4$ & $g^2$ & $C(0,0)$, $C(0,1)$ \\[2mm]
		$\Z_2\times\Z_2$ & $1$ & $P_{\Z_2\times\Z_2}(0,0,0)$, $P_{\Z_2\times\Z_2}(0,0,1)$,
		$P_{\Z_2\times\Z_2}(1,0,0)$, $P_{\Z_2\times\Z_2}(1,\beta_0,0)$ \\[2mm]
		$\Z_2\times\Z_2$ & $g$ & $D(0,0,0)$, $D(0,0,1)$, $D(1,0,0)$ \\[2mm]
		$\Z_2$ & $1$ & $\mathcal T_0=H_0(0,0,0,0,0,0)$,
		$\mathcal T_1=H_0(0,0,0,0,0,1)$;
		$\mathcal E(r,d)=H_0(1,r,d,0,0,0)$,
		$(r,d)\in\{0,1\}^2$ \\[2mm]
		$\Z_2$ & $g$ & $H_1(0,0,0)$, $H_1(0,0,1)$, $H_1(1,0,0)$ \\[2mm]
		\bottomrule
		\caption{Individual representatives of dimension $16$}
		\label{tab:isolated-classes}
	\end{xltabular}
	
	\begin{xltabular}{\textwidth}{ccXl}
		\toprule
		$G$ & $s$ & Family & Parameter equivalence \\
		\midrule
		\endhead
		$\Z_4$ & $1$ & $\mathcal N_1^{(4)}(\lambda)
		=P_{\Z_4}(0,\beta_1,\lambda)$ & $\lambda=\lambda'$ \\[2mm]
		$\Z_4$ & $g^2$ & $\cL(\lambda)=C(1,\lambda)$
		& $\lambda=\lambda'$ \\[2mm]
		$\Z_2\times\Z_2$ & $1$ & $\mathcal P(t,u)=P_{\Z_2\times\Z_2}(0,\beta_t,u)$
		& $(t,u)\sim(t',u')$ \\[2mm]
		$\Z_2\times\Z_2$ & $g$ & $\mathcal M_1(\lambda)=D(0,1,\lambda)$
		& $\lambda=\lambda'$ \\[2mm]
		$\Z_2$ & $1$ & $\mathcal N_1(t)=H_0(0,0,1,0,0,t)$
		& $t^3=(t')^3$ \\[2mm]
		$\Z_2$ & $1$ & $\mathcal N_2(t)=H_0(0,0,t,0,1,0)$
		& $t=t'$ \\[2mm]
		$\Z_2$ & $1$ & $\mathcal N_3(t)=H_0(0,0,t,1,0,0)$
		& $t^5=(t')^5$ \\[2mm]
		$\Z_2$ & $g$ & $\mathcal M_3(t)=H_1(0,1,t)$
		& $t^3=(t')^3$ \\[2mm]
		\bottomrule
		\caption{Parameter families of dimension $16$}
		\label{tab:parameter-families}
	\end{xltabular}

	All single parameters range over $\K$, and $(t,u)$ ranges over $\K^2$.
	The relation $\sim$ in the third row is the equivalence relation
	generated by
	\[
	(t,u)\longmapsto(t+1,u),\qquad
	(t,u)\longmapsto(t^{-1},u/t^3)\quad(t\ne0).
	\]
	
	The parameter sets for $\mathcal N_1$, $\mathcal N_3$, and
	$\mathcal M_3$ are consequently $\K/U_3(\K)$, $\K/U_5(\K)$,
	and $\K/U_3(\K)$, respectively, where these groups act by multiplication.
	
	The proof of Theorem~\ref{thm:dim16-classification} occupies the
	remainder of this section and is divided according to the dimension
	of the diagram $R$.
	
	In Section~\ref{sec:case-dimR4}, we treat the case $\dim R=4$. We
	determine the possible eight-dimensional Hopf subalgebras
	$L\subseteq H$ with $L_0=H_0$, realize $H$ as a normalized quadratic
	extension of $L$ by Theorem~\ref{thm:classification-index-two}, and
	solve the corresponding extension and isomorphism problems. This yields
	precisely the rows of Tables~\ref{tab:isolated-classes} and
	\ref{tab:parameter-families} with group $\Z_4$ or $\Z_2\times\Z_2$.
	
	Section~\ref{sec:case-dimR8} treats the case $\dim R=8$ in the same
	way and yields the rows with group $\Z_2$. The reconstruction results
	in the two cases show that every displayed presentation defines a
	pointed Hopf algebra of dimension $16$ with the required diagram,
	while the isomorphism criteria established there give exactly the
	parameter identifications and non-isomorphism assertions stated in
	the theorem.

	\section{The case \(\dim R=4\): base algebras and extensions by \(y\)}\label{sec:case-dimR4}
	
	Throughout this section, $\K$ is algebraically closed of characteristic
	$2$. Except in Lemma~\ref{lem:dim4-pointed} and
	Corollary~\ref{cor:order-two-base}, $H$ denotes a pointed Hopf algebra of
	dimension $16$ with one-dimensional infinitesimal braiding and a
	four-dimensional diagram $R$ which is not a Nichols algebra.
	
	\begin{lem} \cite[Theorem 3.3]{WW}\label{lem:dim4-pointed}
		There are six isomorphism classes of four-dimensional non-connected
		pointed Hopf algebras over $\K$: the group algebras $\K[\Z_4]$ and
		$\K[\Z_2\times\Z_2]$, and the following four algebras. In each of the
		four presentations $g^2=1$, $\Delta(g)=g\otimes g$, and
		$\epsilon(x)=0$:
		\begin{enumerate}
			\item $[g,x]=0$, $x^2=\epsilon x$, $\epsilon\in\{0,1\}$, with
			$\Delta(x)=x\otimes1+1\otimes x$;
			\item $[g,x]=\epsilon(1+g)$, $x^2=\epsilon x$,
			$\epsilon\in\{0,1\}$, with $\Delta(x)=x\otimes1+g\otimes x$.
		\end{enumerate}
	\end{lem}
	%
	\subsection{The eight-dimensional base Hopf algebras}
	
	\begin{defn}\label{def:base-4dim}
		Let $G$ be $\Z_4=\langle g\rangle$ or
		$\Z_2\times\Z_2=\langle g,h\rangle$, and choose $s\in G$ as below.
		The base algebras $L$ have the group relations of $G$ and
		\[
		\Delta(t)=t\otimes t\quad(t\in G),\qquad
		\Delta(x)=x\otimes1+s\otimes x,\qquad \epsilon(x)=0.
		\]
		Their remaining relations, in normalized generators, are:
		\begin{enumerate}
			\item $s=1$: $[t,x]=0$ for every $t\in G$, and
			$x^2=\epsilon x$, $\epsilon\in\{0,1\}$;
			\item $G=\Z_4$, $s=g$: $[g,x]=a g(1+g)$,
			$x^2=ax$, $a\in\{0,1\}$;
			\item $G=\Z_4$, $s=g^2$: $[g,x]=r g(1+g^2)$,
			$x^2=0$, $r\in\{0,1\}$;
			\item $G=\Z_2\times\Z_2$, $s=g$:
			$[g,x]=a(1+g)$, $[h,x]=b h(1+g)$, $x^2=ax$,
			where $a\in\{0,1\}$ and $b\in\K$.
		\end{enumerate}
		In the last case $[g,h]=0$ is part of the presentation. Each algebra
		has basis $\{tx^j:t\in G,\ j=0,1\}$ and dimension $8$.
	\end{defn}
	
	\begin{lem}\label{lem:base-location}
		The group $\G(H)$ has order $4$. The Hopf subalgebra
		$L=\langle H_1\rangle=\langle \G(H),x\rangle$ has dimension $8$ and
		admits a presentation in Definition~\ref{def:base-4dim}.
		For $t\in \G(H)$,
		\begin{equation}\label{eq:section4-primitive-spaces}
			\Pp_{1,t}(H)=\K(1-t)+
			\begin{cases}\K x,&t=s,\\0,&t\ne s.\end{cases}
		\end{equation}
		When $s=1$ the first summand for $t=s$ is zero.
	\end{lem}
	\begin{proof}
		By $\dim H=|\G(H)|\dim R$ and $\dim H=16$, $\dim R=4$, we obtain
		$|\G(H)|=4$. Since $\Char\K=2$, every multiplicative character of $\G(H)$ is trivial.
		The degree-one part yields a skew-primitive lift $x$ with support $s$,
		and Taft--Wilson Theorem gives \eqref{eq:section4-primitive-spaces}.
		
		Let $G=\G(H)$. If $s=1$, then $x$ is primitive, so $[t,x]=0$ and
		$x^2=\epsilon x$. If $s\ne1$, comparing coproducts gives
		\[
		[t,x]=\chi(t)t(1+s),\qquad
		x^2=\chi(s)x+d(1+s^2),
		\]
		where $\chi:G\to(\K,+)$ is a group homomorphism. Consistency of the
		square relation yields
		\[
		\chi(t)\bigl(\chi(t)+\chi(s)\bigr)(1+s^2)=0.
		\]
		These formulas give the four cases in Definition~\ref{def:base-4dim}.
		For $s=g$ of order $4$, replacing $x$ by $x+c(1+g)$ changes $d$ to
		$d+c^2+\chi(g)c$, so $d$ can be removed; rescaling $x$ normalizes
		nonzero coefficients to $1$. For $\Z_2\times\Z_2$, $\chi(h)$ remains arbitrary.
		
		The relations reduce any word to $tx^j$ with $j=0,1$, whose leading
		terms in $\gr H$ are independent, hence $\dim L=8$. Conversely, each
		presentation in Definition~\ref{def:base-4dim} satisfies the Diamond
		Lemma, has coradical $\K[G]$, and gives a pointed Hopf algebra of
		dimension $8$.
	\end{proof}

	\subsection{The relevant twisted Hochschild cohomology groups}
	
	\begin{pro}\label{pro:H2-base}
		For every base algebra $L$ in Definition~\ref{def:base-4dim}, for every $t\in G$,
		\[
		\HH^2_{1,t}(L)=0\quad(t\ne s^2),\qquad
		\HH^2_{1,s^2}(L)=\K[\omega_s],\qquad \omega_s=xs\otimes x.
		\]
		In particular, the cocycles for the five support cases are respectively
		$x\otimes x$, $xg\otimes x$, $xg^2\otimes x$,
		$x\otimes x$, and $xg\otimes x$.
	\end{pro}
	\begin{proof}
		The identity
		$\Delta(xs)=xs\otimes s+s^2\otimes xs$
		shows that $\omega_s=xs\otimes x$ is a $2$-cocycle. Since
		\[
		\Delta(tx)=tx\otimes t+ts\otimes tx
		\]
		for all $t\in G$, no coboundary has a term with $x$ in both tensor
		factors. Hence $[\omega_s]\neq0$.
		
		By the coalgebra Hochschild cohomology identification
		\eqref{eq:coalgebra-ext}
		from Section~\ref{sec:prelim}, it suffices to compute the
		corresponding Ext spaces over the dual algebra $L^*$. As a
		coalgebra, $L$ has vertices $G$ and one arrow $tx$ from $t$ to
		$ts$; hence $L^*$ is the path algebra of this quiver modulo all
		paths of length at least $2$. The projective cover of a vertex
		simple has radical the next vertex simple and radical square zero.
		Therefore the degree-two Ext space is one-dimensional exactly for
		the endpoint $s^2$ reached by two arrows, and zero otherwise. Thus
		$\HH^2_{1,s^2}(L)=\K[\omega_s]$ and
		$\HH^2_{1,t}(L)=0$ for $t\neq s^2$.
	\end{proof}

	\begin{lem}\label{lem:section4-y}
		The algebra $H$ admits generators $G,x,y$ such that
		\begin{equation}\label{eq:section4-delta-y}
			\Delta(y)=y\otimes1+s^2\otimes y+xs\otimes x,
			\qquad \epsilon(y)=0.
		\end{equation}
		The monomials $tx^iy^j$, $t\in G$, $i,j\in\{0,1\}$, form a basis.
		Moreover, $[y,L]\subseteq L$, and $L$ is preserved by every Hopf
		isomorphism between two such algebras.
	\end{lem}
	\begin{proof}
		Since $H_0=L_0$ and $\dim H=2\dim L$,
		Theorem~\ref{thm:classification-index-two} shows that $H$ is a
		normalized quadratic extension of $L$. Thus there exists
		$y\in\ker\epsilon_H$ such that $\{1,y\}$ is a free left $L$-basis of
		$H$ and
		\[
		\Delta(y)=y\otimes g+h\otimes y+\omega
		\]
		for some $g,h\in\G(L)$ and $\omega\in L^+\otimes L^+$.
		
		By Lemma~\ref{lem:prelim-extension-injection}, the class of the
		coproduct correction is nonzero in $\HH^2_{g,h}(L)$. Replacing $y$ by
		$g^{-1}y$ (as in Remark~\ref{rmk:unmarked-supports}), we may assume
		$g=1$; the new second support is $g^{-1}h$, which we denote again by
		$h$. The corresponding correction still represents a nonzero class
		in $\HH^2_{1,h}(L)$. Proposition~\ref{pro:H2-base} therefore gives
		\[
		h=s^2
		\]
		and shows that this class is a nonzero scalar multiple of
		$[xs\otimes x]$. After rescaling the generator and adding an element
		of $L^+$ to remove the resulting coboundary (see
		Definition~\ref{def:gauge}), we may choose $y$ so that
		\[
		\Delta(y)=y\otimes1+s^2\otimes y+xs\otimes x.
		\]
		This proves \eqref{eq:section4-delta-y}.
		
		Since $H=L\oplus Ly$ and $\{tx^i:t\in G,\ i=0,1\}$ is a basis of
		$L$, the monomials $tx^iy^j$, $t\in G$, $i,j\in\{0,1\}$, form a
		basis of $H$.
		
		For $a\in L$, 
		\[
		ya=\sigma(a)y+\delta(a).
		\]
		We claim that $\sigma=\id_L$. Let $\chi=\epsilon\sigma$. By
		\eqref{eq:comp1} with $g=1$,
		\[
		\sigma(a)=\sum\chi(a_{(1)})a_{(2)}.
		\]
		Since every multiplicative character of $G$ is trivial,
		$\sigma(t)=t$ for all $t\in G$, while
		\[
		\sigma(x)=x+c,\qquad c=\chi(x)\in\K.
		\]
		If $s\ne1$, equation \eqref{eq:comp2} applied to $x$ gives
		\[
		c(s^2-s^3)=0,
		\]
		and hence $c=0$. If $s=1$, preservation of $x^2=\epsilon x$ gives
		$c^2=\epsilon c$. Thus $c=0$ unless $\epsilon=c=1$. In the latter
		case, equation \eqref{eq:comp3} applied to $x$ gives
		\[
		d^1_{1,1}(\delta(x))=x\otimes x,
		\]
		contradicting Proposition~\ref{pro:H2-base}. Hence $c=0$ in all
		cases, and therefore $\sigma=\id_L$. It follows that
		\[
		[y,L]\subseteq L.
		\]
		
		Finally, $L=\langle H_1\rangle$ is intrinsic, since Hopf
		isomorphisms preserve the coradical filtration and hence $H_1$.
	\end{proof}

	\medskip
	\noindent
	We now determine all possible extensions obtained by adjoining $y$.
	The classification splits according to the support $s$ of $x$.
	\medskip

	\subsection{Primitive support}
	
	\begin{pro}\label{pro:primitive-support-data}
		Let $s=1$ and let $G$ be a group of order $4$. All extensions are
		the algebras $P_G(\epsilon,\beta,\mu)$ with group relations $G$ and
		\begin{gather*}
			[t,x]=0,\qquad [t,y]=\beta(t)tx\quad(t\in G),\qquad [x,y]=0,\\
			x^2=\epsilon x,\qquad y^2=\epsilon y+\mu x,
		\end{gather*}
		where $\epsilon\in\{0,1\}$, $\mu\in\K$, and
		$\beta:G\to(\K,+)$ is a homomorphism satisfying
		\begin{equation}\label{eq:primitive-beta-condition}
			\epsilon\beta(t)\bigl(\epsilon+\beta(t)\bigr)=0\quad(t\in G).
		\end{equation}
		The coproduct is given by \eqref{eq:section4-delta-y} with $s=1$.
		Each presentation defines a Hopf algebra with the PBW basis of
		Lemma~\ref{lem:section4-y}.
	\end{pro}
	\begin{proof}
		The coproduct gives $[x,y]=\rho x$ and
		$[t,y]=\beta(t)tx$, since these elements belong to
		$\Pp(H)$ and $\Pp_{t,t}(H)$ respectively.
		The group relations imply additivity of $\beta$.
		The square coproduct gives $y^2=\epsilon y+\mu x$.
		Applying the derivation $\delta=[y,-]$ to $x^2=\epsilon x$ gives
		$\epsilon\rho=0$, and comparing $\delta^2(x)$ with
		$[\mu x,x]+\epsilon\delta(x)$ gives
		$\rho^2+\epsilon\rho=0$. Hence $\rho=0$.
		For $t\in G$,
		\[
		\delta^2(t)+\epsilon\delta(t)
		=\epsilon\beta(t)\bigl(\beta(t)+\epsilon\bigr)tx.
		\]
		This proves \eqref{eq:primitive-beta-condition}.
		
		Conversely, $\delta(t)=\beta(t)tx$, $\delta(x)=0$ defines a derivation
		of $L$. With $u=\mu x$, $v=\epsilon$, it satisfies
		$\delta^2=[u,-]+v\delta$ and $\delta(u)=\delta(v)=0$ exactly under
		\eqref{eq:primitive-beta-condition}. The coproduct identities follow by
		substitution on $t,x,y$. Proposition~\ref{pro:reconstruction} gives
		the Hopf algebra and its dimension. Its graded coalgebra in degrees
		$0,1,2,3$ has the basis described in Lemma~\ref{lem:section4-y}, so
		its infinitesimal braiding is one-dimensional.
	\end{proof}
	
	\begin{pro}\label{pro:primitive-support-isomorphisms}
		Two algebras $P_G(\epsilon,\beta,\mu)$ and
		$P_G(\epsilon',\beta',\mu')$ are isomorphic if and only if there exist
		$\theta\in\Aut(G)$, $A\in\K^\times$, and $B\in\K$ such that
		\begin{gather*}
			\epsilon=A\epsilon',\qquad
			\beta(t)=A\beta'(\theta(t))\quad(t\in G),\\
			A\mu=A^4\mu'+\epsilon'B^2+\epsilon B.
		\end{gather*}
		The corresponding map is
		$t\mapsto\theta(t)$, $x\mapsto Ax'$, $y\mapsto A^2y'+Bx'$.
	\end{pro}
	\begin{proof}
		An isomorphism induces $\theta$ on the group-like elements.
		Since $\Pp(H')=\K x'$, its value on $x$ is $Ax'$.
		The coproduct of $y$ then forces $A^2y'+Bx'$ as its image.
		The square of $x$, the commutators with $G$, and the square of $y$
		give the displayed equations. If $\epsilon\ne0$, the first equation
		forces $\epsilon=\epsilon'=A=1$; otherwise both epsilons are zero.
		Thus the coefficient of $y'$ in the square equation also agrees.
		Conversely these equations preserve every relation and the coproduct;
		the triangular map on the PBW basis is invertible.
	\end{proof}
	
	\begin{lem}\label{lem:case-m=0}
		For $G=\Z_4=\langle g\rangle$ and $s=1$, representatives are
		\[
		P_G(0,0,0),\quad P_G(0,0,1),\quad
		P_G(1,0,0),\quad P_G(1,\beta_1,0),\quad
		\mathcal N_1^{(4)}(\lambda):=P_G(0,\beta_1,\lambda),
		\]
		where $\beta_1(g)=1$ and $\lambda\in\K$.
		They are pairwise non-isomorphic, and
		$\mathcal N_1^{(4)}(\lambda)\cong\mathcal N_1^{(4)}(\lambda')$
		if and only if $\lambda=\lambda'$.
	\end{lem}
	\begin{proof}
		The homomorphism $\beta$ is determined by $\beta(g)$, which is
		unchanged by every automorphism of $\Z_4$. If $\epsilon=0$ and
		$\beta=0$, a nonzero $\mu$ can be normalized to $1$ by a cube root.
		If $\beta(g)\ne0$, rescale $x,y$ with weights $1,2$ to make
		$\beta(g)=1$; the remaining parameter is $\lambda$.
		If $\epsilon=1$, \eqref{eq:primitive-beta-condition} gives
		$\beta(g)\in\{0,1\}$, and $y\mapsto y+Bx$ removes $\mu$ by
		surjectivity of $B\mapsto B^2+B$.
		Proposition~\ref{pro:primitive-support-isomorphisms} proves all the
		non-isomorphism assertions.
	\end{proof}
	
	\subsection{Nontrivial support over the cyclic group}
	
	\begin{lem}\label{lem:case-m=1}
		For $G=\Z_4=\langle g\rangle$ and $s=g$, there are exactly two
		isomorphism classes, represented by $\mathcal A_a$, $a\in\{0,1\}$:
		\begin{gather*}
			g^4=1,\qquad x^2=ax,\qquad y^2=ay,\qquad [x,y]=0,\\
			[g,x]=a g(1+g),\qquad
			[g,y]=a(xg^2+xg+g^2+g^3).
		\end{gather*}
		Their coproduct is \eqref{eq:section4-delta-y} with $s=g$.
	\end{lem}
	\begin{proof}
		Normalize the base as in Definition~\ref{def:base-4dim}.
		The coproduct of $[g,y]$ first gives
		\[
		[g,y]=a(xg^2+xg+ag^2+ag)+d g(1+g^2).
		\]
		The coproduct of $[x,y]$ then gives
		$[x,y]=\rho x+\lambda(1+g^3)$, where $\rho=a^2+d$.
		Comparison in the square coproduct determines the square modulo
		$\Pp(H)=0$ and gives
		\[
		y^2=ay+\lambda x(1+g^3)+a\rho x.
		\]
		Applying $[y,-]$ to $x^2=ax$ gives $a\rho=0$.
		If $a=0$, the coefficient of $x$ in the overlap $y^2x$ is $\rho^2$,
		so $\rho=0$ in this case as well. Therefore $d=a^2=a$, and
		\begin{gather*}
			[g,y]=a(xg^2+xg+g^2+g^3),\qquad
			[x,y]=\lambda(1+g^3),\\
			y^2=ay+\lambda x(1+g^3).
		\end{gather*}
		Put $u=\lambda x(1+g^3)$ and $\delta=[y,-]$ on $L$.
		Reducing with the base relations and the displayed commutators gives
		\[
		\delta(u)=\lambda^2(1+g^2).
		\]
		The overlap $y^3$, equivalently $\delta(u)=0$ in
		Proposition~\ref{pro:associative-data}, forces $\lambda=0$, since the
		group elements $1,g^2$ are linearly independent in $H$.
		
		For $\lambda=0$, the displayed values of $\delta(g)$ and $\delta(x)=0$
		preserve the relations of $L$ and satisfy $\delta^2=a\delta$.
		With $u=0$, $v=a$, all the reconstruction equations hold. This proves
		existence and the PBW basis. The algebra $\mathcal A_0$ is commutative,
		whereas $[g,x]\ne0$ in $\mathcal A_1$; they are not isomorphic.
	\end{proof}

	\begin{lem}\label{lem:case-m=2theta}
		For $G=\Z_4=\langle g\rangle$ and $s=g^2$, all extensions have
		presentations $C(r,\lambda)$:
		\begin{gather*}
			g^4=1,\quad x^2=0,\quad [g,x]=r gq,\quad
			[g,y]=r xgq+r^2gq,\\
			[x,y]=\lambda q,\qquad y^2=\lambda xq,
			\qquad q=1+g^2,
		\end{gather*}
		with $r,\lambda\in\K$ and coproduct
		\eqref{eq:section4-delta-y} for $s=g^2$.
		Representatives are $C(0,0)$, $C(0,1)$, and
		$C(1,\lambda)= \cL(\lambda)$, $\lambda\in\K$.
	\end{lem}
	\begin{proof}
		Here $x^2=0$ and $\Pp(H)=0$. Subtracting
		$r xgq+r^2gq$ from $[g,y]$ gives a $(g,g)$-primitive element,
		which is zero. The coproduct of $[x,y]$ gives
		$[x,y]=\rho x+\lambda q$. Comparison in the square coproduct gives
		$y^2=\lambda xq$, while the overlap $xy^2$ gives $\rho^2=0$.
		Thus $\rho=0$.
		
		Conversely set $\delta(g)=r xgq+r^2gq$, $\delta(x)=\lambda q$,
		$u=\lambda xq$, $v=0$. The identities $q^2=0$ and $[q,x]=0$
		give $\delta(q)=0$, $\delta(u)=0$, and $\delta^2=[u,-]$.
		Applying $\delta$ to $g^4-1$, $x^2$, and $[g,x]-rgq$ gives zero.
		The coproduct identities on $g,x$ and the square relation hold by
		substitution. Reconstruction gives dimension $16$.
		Under $x\mapsto A x$, $y\mapsto A^2y$, the parameters change to
		$Ar$ and $A^3\lambda$. Hence $r\ne0$ can be normalized to $1$,
		and when $r=0$ a nonzero $\lambda$ can be normalized to $1$.
	\end{proof}
	
	\begin{pro}\label{pro-iso-Zn}
		The cyclic-group representatives in Lemmas~\ref{lem:case-m=0},
		\ref{lem:case-m=1}, and \ref{lem:case-m=2theta} are pairwise
		non-isomorphic, except that equal parameters give the same algebra.
		In particular
		\[
		\cL(\lambda)\cong\cL(\lambda')
		\quad\Longleftrightarrow\quad\lambda=\lambda'.
		\]
	\end{pro}
	\begin{proof}
		The support, up to automorphisms of $G$, separates the three lemmas.
		The primitive-support case was proved above. The two algebras with
		support of order $4$ were separated in Lemma~\ref{lem:case-m=1}.
		For support $g^2$, the vanishing of $[g,x]$ separates $r=0$ and
		$r\ne0$, and commutativity separates $C(0,0)$ and $C(0,1)$.
		
		Let $\Phi:C(1,\lambda)\to C(1,\lambda')$ be an isomorphism.
		Then $\Phi(g)=(g')^e$, $e\in\{1,3\}$, and the skew-primitive
		space and $[g,x]=g(1+g^2)$ give
		\[
		\Phi(x)=x'+c q',\qquad q'=1+(g')^2.
		\]
		Since $\Pp(H')=0$, the coproduct uniquely determines
		\[
		\Phi(y)=y'+c x'q'+c^2q'.
		\]
		The element $q'$ commutes with $x',y'$ and has square zero.
		Consequently $[\Phi(x),\Phi(y)]=\lambda' q'$, which must equal
		$\lambda q'$. Hence $\lambda=\lambda'$. The identity map proves
		the converse.
	\end{proof}
	
	\subsection{The direct product group $\Z_2\times\Z_2$}
	
	\begin{lem}\label{lem:42n-Classification-2}
		Let $G=\langle g,h\mid g^2=h^2=1,\ gh=hg\rangle$.
		For support $1$ the representatives are
		\begin{enumerate}
			\item $P_G(0,0,0)$, $P_G(0,0,1)$, and $P_G(1,0,0)$;
			\item $P_G(1,\beta_0,0)$ with $\beta_0(g)=1$, $\beta_0(h)=0$;
			\item $\mathcal P(t,u)=P_G(0,\beta_t,u)$,
			where $\beta_t(g)=1$, $\beta_t(h)=t$ and $(t,u)\in\K^2$.
		\end{enumerate}
		For support $g$, all extensions have presentations $D(a,b,\lambda)$:
		\begin{gather*}
			[g,h]=0,\quad g^2=h^2=1,\quad q=1+g,\quad x^2=ax,\\
			[g,x]=a q,\quad [h,x]=b hq,\quad
			[g,y]=a xq+a^2q,\quad [h,y]=b xhq+b^2hq,\\
			[x,y]=\lambda q,\qquad y^2=ay+\lambda xq,
		\end{gather*}
		where $a\in\{0,1\}$, $b,\lambda\in\K$, and
		\begin{equation}\label{eq:D-condition}
			a(b^2+b)=0.
		\end{equation}
		The coproduct is \eqref{eq:section4-delta-y} with $s=g$.
		Representatives for this support are
		\[
		D(0,0,0),\qquad D(0,0,1),\qquad
		\mathcal M_1(\lambda):=D(0,1,\lambda),\qquad D(1,0,0).
		\]
		Every displayed presentation has dimension $16$.
	\end{lem}
	\begin{proof}
		For primitive support, Proposition~\ref{pro:primitive-support-data}
		applies. If $\epsilon=0$ and $\beta=0$, normalize $\mu$ to $0$ or $1$.
		If $\epsilon=0$ and $\beta\ne0$, choose a group basis with
		$\beta(g)\ne0$ and rescale to obtain $\beta(g)=1$.
		If $\epsilon=1$, the image of $\beta$ is contained in $\{0,1\}$;
		all its nonzero possibilities are equivalent under $\Aut(G)$, and
		the Artin--Schreier change in $y$ removes $\mu$.
		
		For support $g$, use the last base presentation of
		Definition~\ref{def:base-4dim}. The coproduct equations give
		\[
		[g,y]=a xq+a^2q,\qquad [h,y]=b xhq+b^2hq.
		\]
		The differences from these expressions are respectively in
		$\Pp_{g,g}(H)$ and $\Pp_{h,h}(H)$, both of which vanish.
		The remaining coproduct and square comparisons give
		$[x,y]=\rho x+\lambda q$ and $y^2=ay+\lambda xq$;
		the overlaps with $x^2$ and $y^2$ force $\rho=0$.
		Set $\delta=[y,-]$, $u=\lambda xq$, and $v=a$.
		Applying $\delta$ to the base relation $[h,x]=b hq$ leaves
		$a(b^2+b)hq$, so \eqref{eq:D-condition} is necessary.
		
		Conversely the displayed values of $\delta(g),\delta(h),\delta(x)$
		preserve the base relations under \eqref{eq:D-condition}.
		Using $q^2=0$, one obtains
		$\delta^2=[u,-]+a\delta$, $\delta(u)=0$, and $\delta(a)=0$.
		The coproduct identities hold on $g,h,x$ and on $y^2$.
		Proposition~\ref{pro:reconstruction} gives the PBW basis and dimension.
		When $a=0$, rescaling normalizes $b\ne0$ to $1$; if $b=0$, it
		normalizes $\lambda\ne0$ to $1$.
		When $a=1$, \eqref{eq:D-condition} gives $b\in\{0,1\}$.
		The generator changes in Proposition~\ref{pro-iso-ZnZm} below reduce
		both values of $b$ and every $\lambda$ to $D(1,0,0)$.
	\end{proof}
	
	\begin{pro}\label{pro-iso-ZnZm}
		For primitive support, two members $\Pp(t,u)$ and
		$\Pp(t',u')$ are isomorphic precisely when their parameters
		are related by the equivalence relation generated by
		\begin{equation}\label{eq:P-parameter-orbits}
			(t,u)\longmapsto(t+1,u),\qquad
			(t,u)\longmapsto(t^{-1},u/t^3)\quad(t\ne0).
		\end{equation}
		The other four primitive-support representatives in
		Lemma~\ref{lem:42n-Classification-2} are pairwise non-isomorphic and
		are not isomorphic to a member of $\Pp$.
		For nontrivial support,
		\[
		\mathcal M_1(\lambda)\cong\mathcal M_1(\lambda')
		\quad\Longleftrightarrow\quad\lambda=\lambda'.
		\]
		If $\mathcal M_2(\lambda,\gamma)$ denotes $D(1,\gamma,\lambda)$,
		then
		\[
		\mathcal M_2(\lambda,\gamma)\cong D(1,0,0)
		\qquad(\lambda\in\K,\ \gamma\in\{0,1\}).
		\]
		Apart from \eqref{eq:P-parameter-orbits}, the representatives in
		Lemma~\ref{lem:42n-Classification-2} are pairwise non-isomorphic.
	\end{pro}
	\begin{proof}
		For $\Pp$, changing $h$ to $gh$ gives the first transformation
		in \eqref{eq:P-parameter-orbits}. Interchanging $g,h$ and rescaling
		$x,y$ by $t^{-1},t^{-2}$ gives the second. These changes generate the
		action of $\Aut(G)=\operatorname{GL}_2(\mathbb F_2)$ followed by the
		normalization $\beta(g)=1$. Proposition~\ref{pro:primitive-support-isomorphisms}
		therefore gives exactly the stated equivalence relation. The same
		proposition separates $\epsilon=0$ from $\epsilon=1$, $\beta=0$ from
		$\beta\ne0$, and the two possibilities for $\mu$ when both vanish.
		
		For an isomorphism between two $\mathcal M_1$ algebras, the support
		forces $g\mapsto g'$, and the group map has $h\mapsto(g')^e h'$,
		$e\in\{0,1\}$. The skew-primitive space and the relation
		$[h,x]=h(1+g)$ force
		\[
		x\longmapsto x'+c q',\qquad
		y\longmapsto y'+c x'q'+c^2q',\qquad q'=1+g'.
		\]
		Here the image of $y$ is fixed by its coproduct, since $\Pp(H')=0$.
		The elements $q',x',y'$ satisfy $[q',x']=[q',y']=0$ and $(q')^2=0$;
		comparison of $[x,y]$ gives $\lambda=\lambda'$.
		
		In $D(1,\gamma,\lambda)$, put
		\[
		X=x+cq,\qquad Y=y+c xq+c^2q.
		\]
		Their coproducts are those of $x,y$, with $x$ replaced by $X$.
		Using $q^2=0$, $[x,q]=q$, and $[y,q]=(x+1)q$, multiplication gives
		\[
		X^2=X,\quad [X,Y]=(\lambda+c^2+c)q,\quad
		Y^2=Y+(\lambda+c^2+c)Xq.
		\]
		The commutators with $g,h$ retain their prescribed form and the
		coefficient $\gamma$. Choose $c$ with $c^2+c=\lambda$.
		Replacing $h$ by $gh$ then changes $\gamma$ to $\gamma+1$, leaving
		the other normalized relations unchanged. These invertible generator
		changes prove the assertion for $\mathcal M_2$.
		
		For the remaining nontrivial-support cases, the vanishing of the
		square on the skew-primitive direction separates $a=0$ and $a=1$.
		Within $a=0$, the vanishing of the conjugation action on that direction
		separates $b=0$ and $b\ne0$; commutativity separates $D(0,0,0)$ and
		$D(0,0,1)$. Primitive and nontrivial support are separated by
		$\dim\Pp(H)$.
	\end{proof}
	
	\begin{thm}\label{thm:dim4-extension-classification}
		Let $H$ be a pointed Hopf algebra of dimension $16$ with one-dimensional
		infinitesimal braiding and four-dimensional non-Nichols diagram.
		If $\G(H)\cong\Z_4$, its isomorphism class is represented in
		Lemmas~\ref{lem:case-m=0}, \ref{lem:case-m=1}, and
		\ref{lem:case-m=2theta}. If $\G(H)\cong\Z_2\times\Z_2$, it is
		represented in Lemma~\ref{lem:42n-Classification-2}.
		The complete parameter identifications are those of
		Propositions~\ref{pro-iso-Zn} and \ref{pro-iso-ZnZm}.
	\end{thm}
	\begin{proof}
		The base $L$ is intrinsic and $\dim L=8$ by
		Lemma~\ref{lem:base-location}, hence preserved by every isomorphism.
		With $L$ fixed, Proposition~\ref{pro:H2-base} computes
		$\HH^2_{1,s^2}(L)=\K[\omega_s]$ and
		$\HH^2_{1,t}(L)=0$ for $t\ne s^2$, so Lemma~\ref{lem:section4-y}
		fixes the coproduct of $y$ and the PBW basis of $H$ over $L$. The
		support cases exhaust the group-automorphism orbits. For
		$\G(H)\cong\Z_4$, Lemmas~\ref{lem:case-m=0}, \ref{lem:case-m=1},
		and~\ref{lem:case-m=2theta} give the isomorphism classes; for
		$\G(H)\cong\Z_2\times\Z_2$, Lemma~\ref{lem:42n-Classification-2}
		does. The criteria of Propositions~\ref{pro-iso-Zn}
		and~\ref{pro-iso-ZnZm} account for all group automorphisms and all
		possible images of $x,y$. Hence the list is exhaustive, with the
		stated identifications.
	\end{proof}

	\begin{cor}\label{cor:order-two-base}
		Let $G=\Z_2=\langle g\rangle$. The eight-dimensional pointed Hopf
		algebras with one-dimensional infinitesimal braiding and a
		four-dimensional non-Nichols diagram are obtained as follows:
		\begin{enumerate}
			\item for primitive support, use the presentation
			$P_G(\epsilon,\beta,\mu)$ and condition
			\eqref{eq:primitive-beta-condition}; its isomorphism criterion is
			Proposition~\ref{pro:primitive-support-isomorphisms};
			\item for support $g$, omit $h$ from $D(a,0,\lambda)$.
			Representatives have $(a,\lambda)=(0,0),(0,1),(1,0)$.
		\end{enumerate}
		Their coproducts are \eqref{eq:section4-delta-y}, and their PBW basis
		is $g^ix^jy^k$, $i,j,k\in\{0,1\}$.
	\end{cor}
	\begin{proof}
		The four-dimensional non-trivial bases are those of
		Lemma~\ref{lem:dim4-pointed}. The coalgebra computation in
		Proposition~\ref{pro:H2-base} and the primitive-support derivation
		calculation apply unchanged. For nontrivial support, the proof for
		$D$ applies after omitting $h$. Rescaling normalizes $\lambda\ne0$
		when $a=0$, and the change $X=x+c(1+g)$,
		$Y=y+cx(1+g)+c^2(1+g)$ removes $\lambda$ when $a=1$.
		The PBW basis has $2\cdot2\cdot2=8$ elements.
	\end{proof}

	\section{The case $\dim R=8$: the second cohomological extension}
	\label{sec:case-dimR8}
	
	Throughout this section, $\K$ is algebraically closed of characteristic
	$2$. The algebra $H$ is a pointed Hopf algebra of dimension $16$ whose
	diagram has dimension $8$ and whose infinitesimal braiding is
	one-dimensional. Thus $\G(H)=\langle g\rangle\cong\Z_2$.
	Write $q=1+g$ and let $s\in\{1,g\}$ be the support of the
	infinitesimal braiding. All generators $x,y,z$ have counit zero. The classification of extensions obtained by adjoining $z$ splits
	according to the support $s$ of $x$.
	
	\subsection{The intrinsic base and its coalgebra}
	
	\begin{lem}\label{lem:section5-base}
		The subalgebra $L=\langle H_2\rangle$ is an eight-dimensional Hopf
		subalgebra, and $L=H_3$ as vector spaces. It has generators $g,x,y$ with
		\begin{equation}\label{eq:section5-base-coproduct}
			\Delta(x)=x\otimes1+s\otimes x,\qquad
			\Delta(y)=y\otimes1+1\otimes y+xs\otimes x.
		\end{equation}
		Its possible presentations are those of
		Corollary~\ref{cor:order-two-base}. Every Hopf isomorphism between
		two algebras $H$ as above preserves $L$.
	\end{lem}
	\begin{proof}
		Let $R$ be the diagram of $H$. Then $R(1)=\K x$ with $x^2=0$ in $R$. Set
		$S=\K 1\oplus\K x=\mathcal B(\K x)$, and let $m\ge2$ be minimal
		with $R(m)\neq S(m)$. By Lemma~\ref{lem:prelim-extension-injection},
		\[
		R(m)/S(m)\hookrightarrow \HH^{2,m}_{1,1}(S).
		\]
		The coalgebra $S$ has terms only in degrees $0$ and $1$, so
		\[
		\HH^{2,m}_{1,1}(S)=0\quad(m\ne2),\qquad
		\HH^{2,2}_{1,1}(S)=\K[x\otimes x].
		\]
		Hence $m=2$, and $R(2)=\K y$ for some $y\in R(2)$ with
		$\bar\Delta_R(y)=x\otimes x$.
		
		Since $\Pp(R)=R(1)$ and  $[x,y]\in R(3)$ is primitive,   $[x,y]=0$. Similarly
		$y^2\in R(4)$ is primitive, so $y^2=0$. Together with $x^2=0$, the
		subalgebra of $R$ generated by $x$ and $y$ is
		\[
		B=\operatorname{span}\{1,x,y,xy\},
		\]
		concentrated in degrees $0,1,2,3$.
		
		Consider the degree-$3$ part of the normalized $2$-cochains of $B$.
		The space $B(1)\otimes B(2)\oplus B(2)\otimes B(1)$ is spanned by
		$x\otimes y$ and $y\otimes x$, and
		\[
		d^2_{1,1}(x\otimes y)=-x\otimes x\otimes x,\qquad
		d^2_{1,1}(y\otimes x)=x\otimes x\otimes x.
		\]
		Thus
		\[
		Z^{2,3}_{1,1}(B)=\K(x\otimes y+y\otimes x)
		=\K\,\bar\Delta(xy)
		=d^1_{1,1}(\K xy),
		\]
		so $\HH^{2,3}_{1,1}(B)=0$. Applying
		Lemma~\ref{lem:prelim-extension-injection} to the pair $B\subseteq R$
		gives $R(3)=B(3)=\K xy$.
		
		Therefore $R$ has one-dimensional homogeneous components in degrees
		$0,1,2,3$, with basis $1,x,y,xy$. Since $\dim H=16$ and
		$|\G(H)|=2$, the coradical filtration satisfies
		$H_n/H_{n-1}\cong R(n)\otimes\K[\G(H)]$, hence
		\[
		\dim H_3=|\G(H)|\sum_{i=0}^{3}\dim R(i)=2\cdot4=8.
		\]
		
		We write $G=\G(H)$ below. Choose lifts $x\in H_1$ and $y\in H_2$ whose leading terms are the
		displayed generators of $R(1)$ and $R(2)$. The eight elements
		\[
		\{t,\ tx,\ ty,\ txy\ :\ t\in G\}
		\]
		have distinct leading terms in $\gr H$, hence are linearly
		independent. Since $\dim H_3=8$, they form a basis of $H_3$. In
		particular,
		\[
		H_1=\operatorname{span}\{t,tx:t\in G\},\qquad
		H_2=\operatorname{span}\{t,tx,ty:t\in G\}.
		\]
		
		We claim that $\langle H_2\rangle\subseteq H_3$. Since
		$H_2=\operatorname{span}\{t,tx,ty:t\in G\}$, we have
		$\langle H_2\rangle=\langle G,x,y\rangle$. The lifting relations
		are
		\[
		x^2\in H_1,\qquad [x,y]\in H_2,\qquad y^2\in H_3,
		\]
		and, because $G=\Z_2$ acts trivially on $R(1)$ and $R(2)$ in
		characteristic $2$,
		\[
		[t,x]\in H_0,\qquad [t,y]\in H_1\qquad(t\in G).
		\]
		Using these relations, every word in $G$, $x$, and $y$ can be
		rewritten as a linear combination of the eight basis elements
		$\{tx^iy^j:t\in G,\ i,j\in\{0,1\}\}$ of $H_3$. Hence
		$\langle H_2\rangle\subseteq H_3$. Conversely all eight basis
		elements lie in $\langle H_2\rangle$. Therefore
		\[
		L:=\langle H_2\rangle=H_3
		\]
		as vector spaces, and $\dim L=8$.
		
		Since
		\[
		\Delta(H_2)\subseteq
		H_2\otimes H_0+H_1\otimes H_1+H_0\otimes H_2\subseteq L\otimes L
		\]
		and the antipode preserves the coradical filtration, 
		$S(H_2)\subseteq H_2$ and hence $L$ is a Hopf subalgebra:.
		
		Applying Corollary~\ref{cor:order-two-base} to $L$, we may choose
		normalized generators $g,x,y$ in $L$ so that
		\[
		\Delta(x)=x\otimes1+s\otimes x,\qquad
		\Delta(y)=y\otimes1+1\otimes y+xs\otimes x,
		\]
		and the presentations of $L$ are those listed there. This proves
		\eqref{eq:section5-base-coproduct}.
		
		Finally, every Hopf isomorphism preserves the coradical filtration,
		hence preserves $H_3=L$.
	\end{proof}
	\begin{lem}\label{lem:section5-cohomology}
		Let $L$ be any eight-dimensional algebra in
		Corollary~\ref{cor:order-two-base}, with coproduct
		\eqref{eq:section5-base-coproduct}. Then
		\[
		\HH^2_{1,1}(L)=\K[\omega_s],\qquad
		\HH^2_{1,g}(L)=0,
		\qquad
		\omega_s=xys\otimes x+xs\otimes xy+y\otimes y.
		\]
		The nonzero class has internal degree $4$ for the coalgebra grading
		described below.
	\end{lem}
	\begin{proof}
		By \eqref{eq:section5-base-coproduct} and the relations of
		Corollary~\ref{cor:order-two-base},
		\[
		\Delta(xy)=xy\otimes1+s\otimes xy+x\otimes y+ys\otimes x.
		\]
		For $s=1$, the terms $x^2\otimes x+x\otimes x^2$ cancel in
		characteristic $2$. For $s=g$, the cancellation uses
		$x^2=\epsilon x$ and $[g,y]=\epsilon(x+1)q$. Hence the eight
		elements
		\[
		1,\ g,\ x,\ xg,\ y,\ yg,\ xy,\ xyg
		\]
		form a basis of $L$, and the coproduct makes $L$ the truncated path
		coalgebra on the quiver with vertices $\{1,g\}$ and one arrow from
		each vertex to its translate by $s$, keeping only paths of length
		$\le3$. Assign path length as internal degree.
		
		The dual algebra $L^*$ is the path algebra of the opposite quiver
		modulo paths of length at least $4$. For a vertex $i$, the
		projective cover $P_i$ of the simple module $S_i$ has radical
		filtration with successive composition factors
		\[
		S_i,\ S_{i\cdot s},\ S_{i\cdot s^2},\ S_{i\cdot s^3}.
		\]
		Since $s^2=1$, this sequence is
		\[
		S_i,\ S_{i\cdot s},\ S_i,\ S_{i\cdot s}.
		\]
		Hence
		\[
		\Omega^2(S_i)\cong S_i
		\]
		(and, with the path-length grading,
		$\Omega^2(S_i)\cong S_i(-4)$ up to the grading-shift convention).
		Thus the minimal projective resolution is periodic of period $2$,
		and
		\[
		\operatorname{Ext}^2_{L^*}(S_i,S_j)=
		\begin{cases}
			\K,&j=i,\\
			0,&j\ne i.
		\end{cases}
		\]
		By the identification \eqref{eq:coalgebra-ext},
		\[
		\HH^2_{1,1}(L)=\K,\qquad \HH^2_{1,g}(L)=0.
		\]
		
		Direct substitution of the coproducts of $x,y,xy$ shows that
		\[
		\omega_s=xys\otimes x+xs\otimes xy+y\otimes y
		\]
		satisfies $d^2_{1,1}(\omega_s)=0$. Since
		\[
		\deg(xys\otimes x)=3+1=4,\quad
		\deg(xs\otimes xy)=1+3=4,\quad
		\deg(y\otimes y)=2+2=4,
		\]
		$\omega_s$ is homogeneous of internal degree $4$. The differential
		$d^1$ preserves internal degree, but $L(4)=0$. Therefore no
		coboundary can equal $\omega_s$, and $[\omega_s]$ spans
		$\HH^2_{1,1}(L)$.
	\end{proof}

	\subsection{Extensions with primitive support}
	
	For $s=1$, the base algebra $L=L_0(\epsilon,\beta,\mu)$ has relations
	\begin{gather*}
		g^2=1,\qquad [g,x]=[x,y]=0,\qquad [g,y]=\beta gx,\\
		x^2=\epsilon x,\qquad y^2=\epsilon y+\mu x,
	\end{gather*}
	where $\epsilon\in\{0,1\}$ and
	$\epsilon\beta(\epsilon+\beta)=0$.
	
	\begin{pro}\label{pro:H2-dim8-m0}
		For this base algebra,
		\[
		\HH^2_{1,1}(L)=\K[\omega_0],\qquad
		\omega_0=xy\otimes x+x\otimes xy+y\otimes y,
		\qquad \HH^2_{1,g}(L)=0.
		\]
	\end{pro}
	\begin{proof}
		Apply Lemma~\ref{lem:section5-cohomology} with $s=1$.
	\end{proof}
	
	\begin{lem}\label{lem:m0-dim16-extension}
		If the infinitesimal braiding of $H$ is supported at $1$, a generator
		$z$ can be chosen with
		\begin{equation}\label{eq:section5-primitive-z}
			\Delta(z)=z\otimes1+1\otimes z+\omega_0.
		\end{equation}
		It has filtration degree $4$ and $H=L\oplus Lz$ as a free left
		$L$-module. The monomials $g^ix^jy^kz^\ell$,
		$i,j,k,\ell\in\{0,1\}$, form a basis.
	\end{lem}
	
	\begin{proof}
		Since $L=H_3$, $H_0=L_0$, and $\dim H=2\dim L$,
		Theorem~\ref{thm:classification-index-two} shows that $H$ is a
		normalized quadratic extension of $L$. Thus there exists
		$z\in\ker\epsilon_H$ such that $H=L\oplus Lz$ as a free left
		$L$-module and
		\[
		\Delta(z)=z\otimes g+h\otimes z+\omega
		\]
		for some $g,h\in\G(L)$ and $\omega\in L^+\otimes L^+$.
		
		By Lemma~\ref{lem:prelim-extension-injection}, the class
		$[\omega]\in\HH^2_{g,h}(L)$ is nonzero. Replacing $z$ by $g^{-1}z$
		as in Remark~\ref{rmk:unmarked-supports}, we may normalize the
		support to $(1,g^{-1}h)$. The corresponding coproduct correction
		still represents a nonzero cohomology class.
		Proposition~\ref{pro:H2-dim8-m0} then gives $g^{-1}h=1$ and shows
		that the correction class is a nonzero scalar multiple of
		$[\omega_0]$. After rescaling $z$ and adding an element of $L^+$ to
		remove the resulting coboundary, we may therefore choose $z$ so that
		\[
		\Delta(z)=z\otimes1+1\otimes z+\omega_0.
		\]
		
		Since $\omega_0\in L\otimes L=H_3\otimes H_3$, the preceding
		coproduct formula gives
		\[
		\Delta(z)\in H_0\otimes H+H\otimes H_3,
		\]
		and hence $z\in H_4$. On the other hand, $z\notin L=H_3$, so $z$
		has filtration degree exactly $4$.
		
		Since $H=L\oplus Lz$ and $L$ has basis
		$\{g^ix^jy^k:i,j,k\in\{0,1\}\}$, the monomials
		$g^ix^jy^kz^\ell$, $i,j,k,\ell\in\{0,1\}$, form a basis of $H$.
	\end{proof}
	
	\begin{lem}\label{lem:m0-dim16-relations}
		With the generators above, all relations of $H$ are those of
		$H_0(\epsilon,\beta,\gamma,\lambda,\mu,\nu)$:
		\begin{gather}
			g^2=1,\quad [g,x]=[x,y]=[x,z]=0,\quad [g,y]=\beta gx,
			\label{eq:section5-m0-base}\\
			[g,z]=\beta gxy+\beta^2gy+\gamma gx,\qquad [y,z]=\lambda x,
			\label{eq:section5-m0-comm}\\
			x^2=\epsilon x,\quad y^2=\epsilon y+\mu x,\quad
			z^2=\epsilon z+(\epsilon\mu+\lambda)xy+\mu^2y+\nu x.
			\label{eq:section5-m0-square}
		\end{gather}
		Here $\epsilon\in\{0,1\}$, the other parameters belong to $\K$, and
		the necessary and sufficient restrictions are
		\begin{equation}\label{eq:section5-m0-conditions}
			\begin{cases}
				\beta=0,\quad \mu\lambda=0,&\epsilon=0,\\
				\beta\in\{0,1\},\quad\lambda=0,\quad
				\gamma^2+\gamma=\beta(\mu^2+\mu),&\epsilon=1.
			\end{cases}
		\end{equation}
		Conversely these relations and \eqref{eq:section5-primitive-z} define
		a Hopf algebra of dimension $16$ with the required diagram.
	\end{lem}
	\begin{proof}
		Since $\Pp(H)=\K x$ and
		$\Pp_{g,g}(H)=\K gx$,  comparing coproducts gives
		\[
		[g,z]-\beta gxy-\beta^2gy\in\K gx,\qquad [x,z]=d x.
		\]
		Moreover,
		\[
		\Delta([y,z])=[y,z]\otimes1+1\otimes[y,z]
		+[x,z]\otimes x+x\otimes[x,z].
		\]
		The last two terms cancel, so $[y,z]=\lambda x$.
		The base relations give
		\[
		\omega_0^2=\epsilon\omega_0+
		\epsilon\mu(x\otimes y+y\otimes x)+\mu^2x\otimes x.
		\]
		The terms involving $d$ cancel in the square coproduct. Subtracting
		$\epsilon z+(\epsilon\mu+\lambda)xy+\mu^2y$ from $z^2$ leaves a
		primitive element, giving \eqref{eq:section5-m0-square}.
		The overlaps $zx^2$ and $z^2x$ give $\epsilon d=0$ and
		$d^2+\epsilon d=0$, hence $d=0$.
		
		Put $\delta=[z,-]$ on $L$ and
		$u=(\epsilon\mu+\lambda)xy+\mu^2y+\nu x$, $v=\epsilon$.
		The remaining restrictions are obtained from the following reduced
		overlaps (the last row uses the preceding restrictions):
		\[
		\begin{array}{c|l}
			z g^2&\beta^2(\beta+\epsilon)x\\
			z y^2&\epsilon\lambda x\\
			z^3&\lambda(\epsilon^2\mu+\epsilon\lambda+\mu^2)x\\
			z^2g&\epsilon\bigl(\gamma^2+\gamma+
			\beta(\mu^2+\mu)\bigr)gx.
		\end{array}
		\]
		For $\epsilon=0$, the first row forces $\beta=0$, and the third
		forces $\mu\lambda=0$. For $\epsilon=1$, the first two give
		$\beta\in\{0,1\}$ and $\lambda=0$, and the last gives the remaining
		equation in \eqref{eq:section5-m0-conditions}.
		
		Under these restrictions, the displayed values of $\delta$ preserve
		all relations of $L$ and satisfy
		$\delta^2=[u,-]+v\delta$, $\delta(u)=\delta(v)=0$.
		The coproduct identities on $g,x,y$ and on $z^2$ follow by substitution.
		Proposition~\ref{pro:reconstruction} gives the PBW basis and Hopf
		structure. The associated graded coalgebra has one path in each
		degree $0$ through $7$ at each vertex, so the infinitesimal braiding
		has dimension one and the diagram has dimension $8$.
	\end{proof}
	
	\begin{pro}\label{prop:m0-dim16-iso}
		Use primes for the parameters and generators of a second algebra
		$H_0'$. No isomorphism changes $\epsilon$.
		If $\epsilon=\epsilon'=0$, the two algebras are isomorphic exactly
		when there exist $A\in\K^\times$ and $U\in\K$ such that
		\begin{gather}
			\mu=A^3\mu',\qquad \lambda=A^5\lambda',\qquad
			\gamma=A^3\gamma',\label{eq:section5-m0-iso0}\\
			\nu=A^7\bigl(\nu'+U^4\mu'+U^2\lambda'+U(\mu')^2\bigr).
			\label{eq:section5-m0-nu0}
		\end{gather}
		If $\epsilon=\epsilon'=1$, they are isomorphic exactly when
		$\beta=\beta'$ and there exist $B,C\in\K$ such that
		\begin{gather}
			\mu=\mu'+B^2+B,\qquad
			\gamma=\gamma'+\beta(B^2+B),\label{eq:section5-m0-iso1}\\
			\nu=\nu'+(B^4+B^2)\mu'+B(\mu^2+\mu)+C^2+C.
			\label{eq:section5-m0-nu1}
		\end{gather}
	\end{pro}
	\begin{proof}
		The unique nontrivial group-like element is fixed. The primitive
		space and the coproducts successively force
		\begin{equation}\label{eq:section5-m0-map}
			\Phi(x)=Ax',\quad \Phi(y)=A^2y'+Bx',\quad
			\Phi(z)=A^4z'+A^2B x'y'+B^2y'+Cx'.
		\end{equation}
		The square of $x$ gives $\epsilon=A\epsilon'$, so if either epsilon
		is nonzero both are $1$ and $A=1$.
		
		For $\epsilon=0$, write $B=A^2U$. Substitution in the commutators
		and the squares of $y,z$ gives \eqref{eq:section5-m0-iso0} and
		\eqref{eq:section5-m0-nu0}; the coefficient $C$ is unrestricted.
		For $\epsilon=1$, substitution gives $\beta=\beta'$ and
		\eqref{eq:section5-m0-iso1}. Comparison of the $x'$ coefficient in
		the square of $\Phi(z)$ gives \eqref{eq:section5-m0-nu1}.
		Conversely these equations preserve every defining relation, and
		\eqref{eq:section5-m0-map} preserves the coproduct. Its leading
		coefficients on the PBW basis are nonzero, so it is an isomorphism.
	\end{proof}
	
	\begin{pro}\label{prop:m0-dim16-classification}
		For primitive support, a complete set of representatives is given by
		the following parameter tuples for $H_0$:
		\[
		\begin{array}{c|cccccc}
			&\epsilon&\beta&\gamma&\lambda&\mu&\nu\\\hline
			\mathcal T_0&0&0&0&0&0&0\\
			\mathcal T_1&0&0&0&0&0&1\\
			\mathcal N_1(t)&0&0&1&0&0&t\\
			\mathcal N_2(t)&0&0&t&0&1&0\\
			\mathcal N_3(t)&0&0&t&1&0&0\\
			\mathcal E(r,d)&1&r&d&0&0&0
		\end{array}
		\]
		Here $t\in\K$ and $(r,d)\in\{0,1\}^2$. The only parameter
		identifications are
		\begin{align*}
			\mathcal N_1(t)\cong\mathcal N_1(t')&\iff t^3=(t')^3,\\
			\mathcal N_2(t)\cong\mathcal N_2(t')&\iff t=t',\\
			\mathcal N_3(t)\cong\mathcal N_3(t')&\iff t^5=(t')^5.
		\end{align*}
		The four algebras $\mathcal E(r,d)$ are pairwise non-isomorphic, and
		no isomorphisms occur between different rows.
	\end{pro}
	\begin{proof}
		Suppose first that $\epsilon=0$. If $\mu=\lambda=0$ and $\gamma=0$,
		rescaling normalizes $\nu$ to $0$ or $1$, giving $\mathcal T_0$ and
		$\mathcal T_1$. If $\gamma\ne0$, normalize $\gamma=1$; the residual
		scalings satisfy $A^3=1$ and multiply $\nu$ by $A^7=A$, giving
		$\mathcal N_1$ and its cubic identification.
		If $\mu\ne0$, then $\lambda=0$. Normalize $\mu=1$ and use the
		surjectivity of $U\mapsto U^4+U$ to remove $\nu$.
		The ratio $\gamma/\mu$ is invariant, giving $\mathcal N_2(t)$.
		If $\lambda\ne0$, then $\mu=0$. Normalize $\lambda=1$ and remove
		$\nu$ by $U\mapsto U^2$. The residual scalings satisfy $A^5=1$
		and multiply $\gamma$ by $A^3$, giving the fifth-power identification
		for $\mathcal N_3$.
		
		For $\epsilon=1$, equations \eqref{eq:section5-m0-iso1} show that
		$\beta$ and $d=\gamma+\beta\mu$ are invariant. The parameter
		constraint gives $d^2+d=0$, so $d\in\{0,1\}$.
		Choose $B$ to remove $\mu$ and then $C$ to remove $\nu$ using
		\eqref{eq:section5-m0-nu1}. This gives the four $\mathcal E(r,d)$.
		The vanishing patterns of $\mu,\lambda,\gamma$ in the zero-epsilon
		case and the two invariants in the nonzero-epsilon case establish
		all remaining non-isomorphism assertions.
	\end{proof}
	
	\subsection{Extensions with nontrivial support}
	
	For $s=g$, choose the normalized base
	$L=L_1(\epsilon,\kappa)$ from
	Corollary~\ref{cor:order-two-base}:
	\begin{gather*}
		g^2=1,\quad [g,x]=\epsilon q,\quad [g,y]=\epsilon(x+1)q,\\
		x^2=\epsilon x,\quad [x,y]=\kappa q,\quad
		y^2=\epsilon y+\kappa xq,
	\end{gather*}
	where $(\epsilon,\kappa)\in\{(0,0),(0,1),(1,0)\}$.
	
	\begin{pro}\label{pro:H2-dim8-m1}
		For this base,
		\[
		\HH^2_{1,1}(L)=\K[\omega_1],\qquad
		\omega_1=xyg\otimes x+xg\otimes xy+y\otimes y,
		\qquad \HH^2_{1,g}(L)=0.
		\]
	\end{pro}
	\begin{proof}
		Apply Lemma~\ref{lem:section5-cohomology} with $s=g$.
	\end{proof}
	
	\begin{lem}\label{lem:m1-dim16-coproduct}
		If the infinitesimal braiding of $H$ is supported at $g$, there is a
		generator $z$ of filtration degree $4$ with
		\begin{equation}\label{eq:section5-nonprimitive-z}
			\Delta(z)=z\otimes1+1\otimes z+xyg\otimes x+xg\otimes xy+y\otimes y.
		\end{equation}
		The monomials $g^ix^jy^kz^\ell$, $i,j,k,\ell\in\{0,1\}$, form a
		basis, and $H=L\oplus Lz$ as a free left $L$-module.
	\end{lem}
	\begin{proof}
		Since $L=H_3$, $H_0=L_0$, and $\dim H=2\dim L$,
		Theorem~\ref{thm:classification-index-two} shows that $H$ is a
		normalized quadratic extension of $L$. Thus there exists
		$z\in\ker\epsilon_H$ such that $H=L\oplus Lz$ and
		\[
		\Delta(z)=z\otimes p+q\otimes z+\omega
		\]
		for some $p,q\in\G(L)$ and $\omega\in L^+\otimes L^+$.
		
		By Lemma~\ref{lem:prelim-extension-injection}, the class
		$[\omega]\in\HH^2_{p,q}(L)$ is nonzero. Replacing $z$ by $p^{-1}z$
		as in Remark~\ref{rmk:unmarked-supports}, we may normalize the
		support to $(1,p^{-1}q)$. The corresponding coproduct correction
		still represents a nonzero cohomology class.
		Proposition~\ref{pro:H2-dim8-m1} gives
		\[
		\HH^2_{1,1}(L)=\K[\omega_1],\qquad \HH^2_{1,g}(L)=0,
		\]
		and hence $p^{-1}q=1$. Moreover, the correction class is a nonzero
		scalar multiple of $[\omega_1]$. After rescaling $z$ and adding an
		element of $L^+$ to remove the resulting coboundary, we may choose
		$z$ so that
		\[
		\Delta(z)=z\otimes1+1\otimes z+\omega_1.
		\]
		This proves \eqref{eq:section5-nonprimitive-z}.
		
		Since $\omega_1\in L\otimes L=H_3\otimes H_3$, the preceding
		coproduct formula gives
		\[
		\Delta(z)\in H_0\otimes H+H\otimes H_3,
		\]
		and hence $z\in H_4$. On the other hand, $z\notin L=H_3$, so $z$
		has filtration degree exactly $4$.
		
		Finally, $H=L\oplus Lz$ and $L$ has basis
		$\{g^ix^jy^k:i,j,k\in\{0,1\}\}$, so the monomials
		$g^ix^jy^kz^\ell$, $i,j,k,\ell\in\{0,1\}$, form a basis of $H$.
	\end{proof}
	
	\begin{lem}\label{lem:m1-dim16-relations}
		All extensions of the normalized nontrivial-support bases are
		$H_1(\epsilon,\kappa,\tau)$, obtained by adjoining $z$ with
		\begin{gather}
			[g,z]=\epsilon(yx+y+x+1)q,\label{eq:section5-m1-gz}\\
			[x,z]=\kappa yq+\tau q,\qquad
			[y,z]=\tau xq+\kappa q,\label{eq:section5-m1-comm}\\
			z^2=\epsilon z+\tau yxq+\kappa yq.
			\label{eq:section5-m1-square}
		\end{gather}
		Here $\tau\in\K$ and $\epsilon\tau=0$.
		Together with \eqref{eq:section5-nonprimitive-z}, these are necessary
		and sufficient conditions for a Hopf algebra with the asserted PBW basis.
	\end{lem}
	\begin{proof}
		In this support,
		$\Pp(H)=\Pp_{g,g}(H)=0$ and $\Pp_{1,g}(H)=\K q\oplus\K x$.
		The coproduct comparison for $[g,z]$ gives
		\eqref{eq:section5-m1-gz}, with no $(g,g)$-primitive correction.
		For $[x,z]$, subtracting $\kappa yq$ leaves a $(1,g)$-primitive
		element, so initially
		\[
		[x,z]=\kappa yq+\tau q+\rho x.
		\]
		The coproduct of $[y,z]$ then gives
		$[y,z]=\tau xq+\kappa^2q$. The terms involving $\rho$ cancel.
		The square coproduct determines
		$z^2=\epsilon z+\tau yxq+\kappa^2yq$, since there is no primitive
		correction. The normalized values of $\kappa$ satisfy $\kappa^2=\kappa$.
		The overlap $zx^2$ gives $\epsilon\rho=0$, and the $x$ coefficient
		of $z^2x$ gives $\rho^2+\epsilon\rho=0$, so $\rho=0$.
		Applying $[z,-]$ to $[x,y]=\kappa q$ now leaves
		$\epsilon\tau q$, proving $\epsilon\tau=0$.
		
		Conversely let $\delta=[z,-]$ have the displayed values on $g,x,y$,
		and put $u=\tau yxq+\kappa yq$, $v=\epsilon$.
		Substitution in the base relations gives a derivation of $L$ under
		$\epsilon\kappa=\epsilon\tau=0$, $\kappa^2=\kappa$.
		The identities $q^2=0$ and the base commutators give
		$\delta^2=[u,-]+v\delta$ and $\delta(u)=\delta(v)=0$.
		Substitution on the generators verifies the coproduct compatibility
		and the square identity. Proposition~\ref{pro:reconstruction} gives
		the Hopf algebra and its dimension. Its graded coalgebra has the
		required one-dimensional infinitesimal braiding.
	\end{proof}
	
	\begin{pro}\label{pro:section5-m1-classification}
		For nontrivial support, a complete list is
		\[
		H_1(0,0,0),\qquad H_1(0,0,1),\qquad
		\mathcal M_3(t):=H_1(0,1,t)\ (t\in\K),\qquad H_1(1,0,0).
		\]
		The only parameter identification is
		\[
		\mathcal M_3(t)\cong\mathcal M_3(t')
		\quad\Longleftrightarrow\quad t^3=(t')^3.
		\]
		No isomorphisms occur between different displayed items.
	\end{pro}
	\begin{proof}
		Consider first $\epsilon=0$. In any such algebra the substitutions
		\begin{equation}\label{eq:section5-m1-shift}
			X=x+cq,\quad Y=y+cxq+c^2q,\quad
			Z=z+cxyq+c^2yq+c^4q
		\end{equation}
		preserve the coproducts and all defining relations with the same
		parameters $\kappa,\tau$. Indeed $q$ is central, $q^2=0$, and
		\[
		[X,Z]=\kappa Yq+\tau q,\quad
		[Y,Z]=\tau Xq+\kappa q,\quad
		Z^2=\tau YXq+\kappa Yq.
		\]
		The coproduct correction for $Z$ is the coboundary of
		$cxyq+c^2yq+c^4q$.
		
		An isomorphism fixes $g$ and has $x\mapsto Ax'+cq$ with $A\ne0$.
		Composing with a shift \eqref{eq:section5-m1-shift} removes $c$.
		The absence of primitive elements then forces
		$y\mapsto A^2y'$ and $z\mapsto A^4z'$.
		Comparing $[x,y]$ and $[x,z]$ yields
		\[
		\kappa=A^3\kappa',\qquad \tau=A^5\tau'.
		\]
		If $\kappa=0$, these equations normalize $\tau$ to $0$ or $1$.
		If $\kappa=\kappa'=1$, they require $A^3=1$, so the remaining
		orbits on $\tau$ are exactly those with equal cubes.
		Conversely these scalings preserve all the relations and coproducts.
		
		The condition $\epsilon=1$ forces $\kappa=\tau=0$, giving one
		additional algebra. The commutator $[g,x]$ separates it from
		$\epsilon=0$. Within $\epsilon=0$, the vanishing of $[x,y]$
		separates $\kappa=0$ and $\kappa=1$, and commutativity separates
		the two cases with $\kappa=0$.
	\end{proof}
	
	\begin{thm}\label{thm:section5-classification}
		Every pointed Hopf algebra of dimension $16$ with eight-dimensional
		diagram and one-dimensional infinitesimal braiding is represented in
		Proposition~\ref{prop:m0-dim16-classification} or
		Proposition~\ref{pro:section5-m1-classification}.
		The parameter identifications are exactly those stated there.
		There are nine individual representatives and four parameter families.
	\end{thm}
	\begin{proof}
		By Lemma~\ref{lem:section5-base}, $L$ is intrinsic and
		$\dim L=8$, hence preserved by every isomorphism; moreover
		$L=H_3$ and $\G(H)=\langle g\rangle\cong\Z_2$. With $L$ fixed,
		Propositions~\ref{pro:H2-dim8-m0} and~\ref{pro:H2-dim8-m1}
		give the cohomology classes $\omega_0$ and $\omega_1$, so
		Lemmas~\ref{lem:m0-dim16-extension} and~\ref{lem:m1-dim16-coproduct}
		fix the coproduct of $z$ in the two support cases, and
		Lemmas~\ref{lem:m0-dim16-relations} and~\ref{lem:m1-dim16-relations}
		give the compatible multiplication data. Since every isomorphism
		preserves $L$, the isomorphism criteria of
		Proposition~\ref{prop:m0-dim16-iso} and
		Proposition~\ref{pro:section5-m1-classification} cover all possible
		images of the generators. The two cases satisfy $\dim\Pp(H)=1$ and
		$\dim\Pp(H)=0$, respectively, hence are disjoint. Therefore the
		lists of Proposition~\ref{prop:m0-dim16-classification} and
		Proposition~\ref{pro:section5-m1-classification} are exhaustive,
		with the stated identifications.
	\end{proof}

	\vskip10pt \centerline{\bf ACKNOWLEDGMENT}
	
	\vskip10pt The author wishes to thank Prof. G. Carnovale for her warm hospitality during his visit to Universit\`{a} degli Studi di Padova, and Profs. Quanshui Wu and Xingting Wang for their valuable help and encouragement during his visit to the Shanghai Center for Mathematical Sciences. Part of this work was completed during these visits. The author is partially supported by the National Natural Science Foundation of China (Grant Nos. 11926353 and 12401041).

\end{document}